\documentclass[pdflatex,sn-mathphys-num]{sn-jnl}% Math and Physical Sciences Numbered Reference Style 
\usepackage{graphicx}%
\usepackage{multirow}%
\usepackage{amsmath,amssymb,amsfonts}%
\usepackage{amsthm}%
\usepackage{mathrsfs}%
\usepackage[title]{appendix}%
\usepackage[dvipsnames]{xcolor}%
\usepackage{textcomp}%
\usepackage{manyfoot}%
\usepackage{booktabs}%
\usepackage{algorithm}%
\usepackage{algorithmicx}%
\usepackage{algpseudocode}%
\usepackage{comment}
\usepackage{enumitem}
\usepackage{listings}%
\usepackage{subcaption}
\usepackage[normalem]{ulem}
\definecolor{lightblue}{HTML}{73BFF9}
\newcommand{\IhH}{I_h^H}
\newcommand{\IHh}{I_H^h}

\newcommand{\R}{\normalfont \mathrm{R}}

\newcommand{\A}{\normalfont \mathrm{A}}

\newcommand{\D}{\mathrm{D}}

\newcommand{\acronym}{\mathtt{FLEX-BC-PG}}
\newcommand{\RR}{\mathbb{R}}

\newcommand{\Hi}{\mathcal{H}}

\newcommand{\Id}{\mathrm{Id}}

\DeclareMathOperator*{\argmin}{arg\,min}
\DeclareMathOperator*{\Argmin}{Argmin}

\newcommand{\prox}{\normalfont \textrm{prox}}

\newcommand{\lp}{\normalfont \textrm{lp}}
\newcommand{\vertiii}[1]{{\left\vert\kern-0.25ex\left\vert\kern-0.25ex\left\vert #1 
    \right\vert\kern-0.25ex\right\vert\kern-0.25ex\right\vert}}
    
\definecolor{fuchsia}{rgb}{1.0, 0.0, 1.0}
\theoremstyle{thmstyleone}%
\newtheorem{theorem}{Theorem}[section]% meant for sectionwise numbers
\newtheorem{proposition}[theorem]{Proposition}% 

\theoremstyle{thmstyletwo}%
\newtheorem{remark}[theorem]{Remark}%
\newtheorem{lemma}[theorem]{Lemma}
\newtheorem{assumption}{Assumption}
\theoremstyle{thmstylethree}%
\newtheorem{definition}{Definition}%

\begin{document}

\title[A essentially cyclic multi-block-coordinate proximal gradient  algorithm for non-smooth and non-convex optimization.]{A flexible block-coordinate forward-backward algorithm for non-smooth and non-convex optimization.}


\author[2]{\fnm{Luis} \sur{Briceño-Arias}}\email{luis.briceno@usm.cl}
\equalcont{These authors contributed equally to this work.}
\author[1]{\fnm{Paulo} \sur{Gonçalves}}\email{paulo.goncalves@inria.fr}
\equalcont{These authors contributed equally to this work.}
\author*[1]{\fnm{Guillaume} \sur{Lauga}}\email{guillaume.lauga@ens-lyon.fr}
\author[3]{\fnm{Nelly} \sur{Pustelnik}}\email{nelly.pustelnik@ens-lyon.fr}
\equalcont{These authors contributed equally to this work.}
\author[1]{\fnm{Elisa} \sur{Riccietti}}\email{elisa.riccietti@ens-lyon.fr}
\equalcont{These authors contributed equally to this work.}

\affil*[1]{\orgname{ENS de Lyon, CNRS, Université Claude Bernard Lyon 1, Inria, LIP, UMR 5668}, \orgaddress{69342, Lyon cedex 07, France}}

\affil[2]{\orgdiv{Departamento de Matemática}, \orgname{Universidad Técnica Federico Santa María}, \orgaddress{Santiago, Chile}}

\affil[3]{\orgname{Laboratoire de Physique, ENSL, CNRS UMR 5672}, \orgaddress{F-69342, Lyon, France}}


\abstract{Block coordinate descent (BCD) methods are prevalent in large scale optimization problems due to the low memory and computational costs per iteration, the predisposition to parallelization, and the ability to exploit the structure of the problem. The theoretical and practical performance of BCD relies heavily on the rules defining the choice of the blocks to be updated at each iteration. We propose a new deterministic BCD framework that allows for very flexible updates, while guaranteeing state-of-the-art convergence guarantees on non-smooth non-convex optimization problems. While encompassing several update rules from the literature, this framework allows for priority on updates of particular blocks and 
correlations in the block selection between iterations, which is not permitted under the classical convergent stochastic framework.
%, and allowing BCD methods to  mimic the behavior of multilevel algorithms. 
This flexibility is leveraged in the context of multilevel optimization algorithms and, in particular, in multilevel image restoration problems, where the efficiency of the approach is illustrated.
}




\maketitle
%\tableofcontents

\section{Introduction}\label{sec1}
In this paper we introduce a new flexible block-coordinate algorithm to solve the
separable-structured optimization problem 
\begin{equation}
    \mathbf{\widehat{\mathbf{x}}} \in \Argmin_{\mathbf{x}=(x_1,\dots,x_{L}) \in \Hi} \Psi(\mathbf{x}):= f(\mathbf{x}) + \sum_{\ell = 1}^{L} g_\ell(x_\ell),
    \label{eq5:optim}
\end{equation}
where $\Hi$ is the direct sum of real separable, and finite dimensional Hilbert spaces
$(\Hi_\ell)_{1\le \ell\le L}$, 
$f:\Hi\to (-\infty,+\infty]$ is continuously differentiable, 
and, for every $\ell\in\{1,\ldots,L\}$, $g_\ell:\Hi_\ell \to (-\infty,+\infty]$ is proper and lower semicontinuous.
Without additional assumptions, the minimization problem can be non-smooth and non-convex. 
In image processing, for instance, $f$
usually encodes a fidelity with respect to some observation (e.g., a corrupted image) and functions 
$(g_\ell)_{1\le \ell\le L}$
encode some prior knowledge about components of the parameters to estimate (e.g., the regularity of an image).
In large-scale optimization, when the dimension of $\Hi$ is high, \textit{block-coordinate} (BC) methods are widely used for their low per-iteration computational and memory costs, and their ability to exploit problem separability. The main application of BC approaches is for parameter estimation (e.g., standard machine learning \cite{nesterov2012efficiency,wright2015coordinate,pmlr-v206-larsson23a,friedman2008sparse,xu2013block,xu2017globally,nutini2022let,nutini2015coordinate,salzo_parallel_2022} but also deep learning   \cite{pmlr-v97-zeng19a,gratton2024block,zhang2017convergent}). 
The theoretical convergence of block-coordinate methods depends on the choice of (i) the update rule of blocks $x_1,\ldots, x_L$ and (ii) the function to be minimized for each block.  The update rule can be of two types: stochastic or deterministic.  The function to be minimized for each block can either be $\Psi$ or a linearized version of $\Psi$. The first case  is referred as \textit{Block-Coordinate Descent}  (see, e.g., \cite{Hong17detconv}), which includes the   Gauss-Seidel approach (see \cite{LuoTseng1993} and references therein),  while the second case is \textit{Block-Coordinate Proximal Gradient} (BC-PG), whose particular case when $g_\ell\equiv 0$ reduces to  \textit{Block-Coordinate Gradient Descent}.

In \cite{LuoTseng1993}, the linear convergence of the Block-Coordinate Descent method is proved under the strong convexity of $\Psi$ w.r.t. each block.   However, this approach has two drawbacks: for each block the minimization procedure can be as complex as solving \eqref{eq5:optim} and without strong convexity assumption there is no theoretical guarantees of its convergence \cite{Powell1973}. For these reasons, we focus in this paper on the BC-PG, which offers efficient update of each block due to  linearization and convergence guarantees in a larger setting.

\paragraph{State-of-the-art on BC-PG.}
There exists a wide literature using stochastic activation of blocks in the convex and non-convex setting, see, e.g., \cite{salzo_parallel_2022,briceno2022random,combettes2015stochastic,lin2015accelerated,richtarik2014iteration,richtarik2016parallel,fercoq2015,wright2015coordinate,cadoni2016block,namkoong2017adaptive,lee2019random,sun2021worst,nesterov2012efficiency,Patrascu2015Efficient}. Most of the literature  on stochastic approaches %only 
studies convergence or rate of convergence of the objective function values in expectation \cite{richtarik2014iteration,richtarik2016parallel,fercoq2015,wright2015coordinate}. In \cite{combettes2015stochastic}, the almost sure convergence of the random iterates is proved by using the concept of stochastic Quasi-Féjer sequence, first introduced in \cite{Ermolev1971}. This framework is powerful and can be applied to many types of block-coordinate algorithms (see for instance primal-dual ones in \cite{Alacaoglu2022convergence,briceno2022random, Chambolle2024_SPHDHG}). However, it is thus  far only applicable if $\Psi$ is convex. \ 

The main limitation of stochastic approaches is the random  block selection, which does not offer the flexibility to prescribe the order of block activation, even though it allows for parallel activations. Moreover, a comparison between random and deterministic block activation strategies in the strongly convex setting when $g_{\ell}\equiv 0$ is presented in \cite{nutini2015coordinate}, where the deterministic Gauss-Southwell rule, i.e.,  the strategy that updates at each iteration the block with the largest partial gradient, is shown to be preferable in practice.

Deterministic block-activation algorithms appear in a wide  literature that covers both convex and non convex frameworks, with convergence supported by  theoretical arguments, different from those used in the stochastic setting. %
Two main types of  deterministic rules exist: cyclic and essentially cyclic.%, and their convergence proofs differ from each other.
Cyclic refers to the sequential update of one block after the other until every one of them has been updated once. These updates encompass alternated linearized optimization techniques \cite{bolte2014proximal}. This approach does not permit to update the same component twice in the same cycle.
%Another class of deterministic block update is the 
\textit{Essentially cyclic} rules on the other hand, also called $K$-cyclic \cite{Combettes1997} for $K\geq1$, update  every block at least once in any $K$ consecutive iterations.  %
This rule allows for the activation of a single block multiple times in a same cycle of $K$ iterations.  %

Under \emph{convexity assumptions}, some convergence results are available for cyclic BCD in \cite{LuoTseng1992, Combettes1997, MokhtariGurbuzbalabanRibeiro2018}. These results either cover a different framework than ours (i.e., \eqref{eq5:optim}) or a more restricted one. 
More precisely, when $g_\ell$ is the indicator function of an interval and $f$ is coordinatewise strongly convex, the linear convergence of the  BC-PG with cyclic updates is proved in \cite{LuoTseng1992}. In \cite{Combettes1997}, the convergence of a convex feasibility problem with several updating rules involving cyclic and essentially cyclic approaches is obtained. In \cite{LatafatThemelisPatrinos2022}, linear convergence rates of the functional values and iterates are derived under strong convexity assumptions and when the separability is in the smooth component. A similar strategy is used in \cite{MokhtariGurbuzbalabanRibeiro2018} for minimizing the average of a finite number of strongly convex functions.

BC-PG with Gauss-Southwell rule also belongs to the class of deterministic block-activation algorithms but to the best of our knowledge the results are based on convergence of the values of the function and not of the iterates \cite{nutini2015coordinate,Friedlander2020}.

 Guarantees of convergence of cyclic/essentially cyclic BCD have been investigated in the \emph{nonconvex setting} for instance in \cite{bolte2014proximal,chouzenoux2016block,xu2017globally} under Kurdyka-Łojasiewicz (KŁ) or Łojasiewicz properties/inequalities \cite{attouch2009convergence,attouch2010proximal,bolte2010characterizations,bolte2014proximal}. The authors proved the convergence to a minimizer or a critical point of $\Psi$, and derived the rate of convergence of the sequence of iterates under specific assumptions on the desingularizing function. Among the existing literature, the work most closely related to ours is \cite{chouzenoux2016block}, where the authors examine convergence analysis in the essentially cyclic case. However, their study focuses on updating a single block at each iteration, which does not encompass the parallel essentially cyclic framework that allows for the simultaneous updating of multiple blocks. %

\paragraph{BC-PG approaches for image reconstruction.} 

If BC-PG approaches are widely used in hyperparameter estimation (e.g., machine learning \cite{nesterov2012efficiency,wright2015coordinate,pmlr-v206-larsson23a,friedman2008sparse,xu2013block,xu2017globally,nutini2022let,nutini2015coordinate,salzo_parallel_2022} and deep learning   \cite{pmlr-v97-zeng19a,gratton2024block,zhang2017convergent}),  its application is less straightforward in image reconstruction. Indeed, the image reconstruction is modeled by a minimization problem of the form \eqref{eq5:optim}, whose global structure is not handled well by BC-PG methods that operate through local (patch-wise or pixel-wise) updates \cite{wright2015coordinate}. BC-PG is more effective for images dominated by local information, such as astronomical images, which are composed mostly of a black background with sparse point sources relative to the image size \cite{sun_block_2019}.
For more general images, the structure of the optimization problem can be exploited to design BC-PG schemes. For instance in \cite{pascal_block-coordinate_2018} the authors divide the image in lattices in order to benefit from a reduced size allowing for faster convergence. These lattices are inherited from the structure of the total variation regularization, which penalizes the difference of the value of a given pixel with its neighbor on the right, and its neighbor below. Therefore, grouping pixels by selecting one every other two rows and every other two columns define four independent groups of pixels that cover the entirety of the image. %
In \cite{chouzenoux2016block}, blocks are defined by combining wavelet coefficients locally in an overcomplete dictionary. %

In image processing, block strategies have mainly been implemented to manage large volumes of data. However, they do not necessarily accelerate the solution process for a given data volume. For example, the forward-backward algorithm may remain faster than the BC-PG algorithm if all the data is accessible (see, e.g., \cite{onose2016scalable}). 
\paragraph{Multilevel algorithms for image reconstruction.} There exists another class of algorithms that exploit the structure of the minimization problem and have demonstrated their effectiveness on imaging problems: multilevel algorithms \cite{parpas2017,javaherian2017,hovhannisyan2019fast,fung2020,buccini2020multigrid,plier2021}. Such procedures can speed-up algorithms in the convex framework \cite{lauga2022fista,lauga_iml_2023,lauga2024radio} but also converge towards a better solution in the non-convex setting \cite{refNils,calandra2021high}. A multilevel algorithm tackles high dimensional optimization problems by defining a hierarchy of smaller dimensional approximations of the original problem and by alternating optimization steps on this hierarchy. This approach, similarly to BC-PG, relies on a problem decomposition to perform iterations at lower complexity.  %

As compared to BC-PG,  multilevel methods are much faster %
for imaging problems \cite{lauga_iml_2023}.  However, the notion of coherence between approximations is at the core of the multilevel strategy and is not used for BC-PG algorithms \cite{lauga_iml_2023}. Finally, the convergence for multilevel methods is limited to the convex possibly non-smooth setting or the non-convex smooth setting, while the convergence of BC-PG methods is established both in the convex and the non-convex setting, without smoothness assumptions.

\paragraph{Contribution: Flexible block-coordinate proximal gradient approach.} 

We propose a new BC-PG algorithm that relies on a parallel and essentially cyclic rule, which enables us to fully exploit the %hierarchical 
structure of the optimization problem, whenever it exists. We will refer to it as $\acronym$, which stands for \textit{Flexible BC-PG}.

$\acronym$ can trigger parallel activation of the blocks. Such parallel updates allow for instance to update blocks that the structure of the problem groups together in their contribution to the objective function %
(like with Total Variation in \cite{pascal_block-coordinate_2018}) or share a position in a hierarchy (such as the detail coefficients of a wavelet transform \cite{mallat1999wavelet}). For instance, with respect to the existing cyclic update framework, $\acronym$, allows for varying the size of the blocks along the iterations, which is of interest since using larger block size is beneficial to speed up the optimization \cite{nutini2022let}, but may be more costly. Alternating between the two thus allows to benefit from the advantages of both. %

Our first contribution is to prove that the convergence of $\acronym$ to a critical point of the optimization problem is guaranteed. To the best of our knowledge, no BC-PG algorithm was shown to converge to a critical point of a non-smooth and non-convex optimization problem while allowing parallel updates in a non-stochastic setting. %
Our $\acronym$ algorithm is thus the first that has this property, while having state-of-the-art convergence guarantees (i.e., decrease of objective function values and convergence to a critical point in a non-convex setting)%
We display in Figure \ref{fig:updates_scheme} some of the update schemes covered by our framework. 

\begin{figure*}
    \centering
        \begin{subfigure}[b]{\textwidth}
        \renewcommand{\arraystretch}{0.5}
        \setlength{\tabcolsep}{1pt}
        \begin{tabular}{cc}
          Cyclic   & Cyclic reshuffled  \\
          \includegraphics[trim={21em 16em 25em 20em},clip,width=0.5\textwidth]{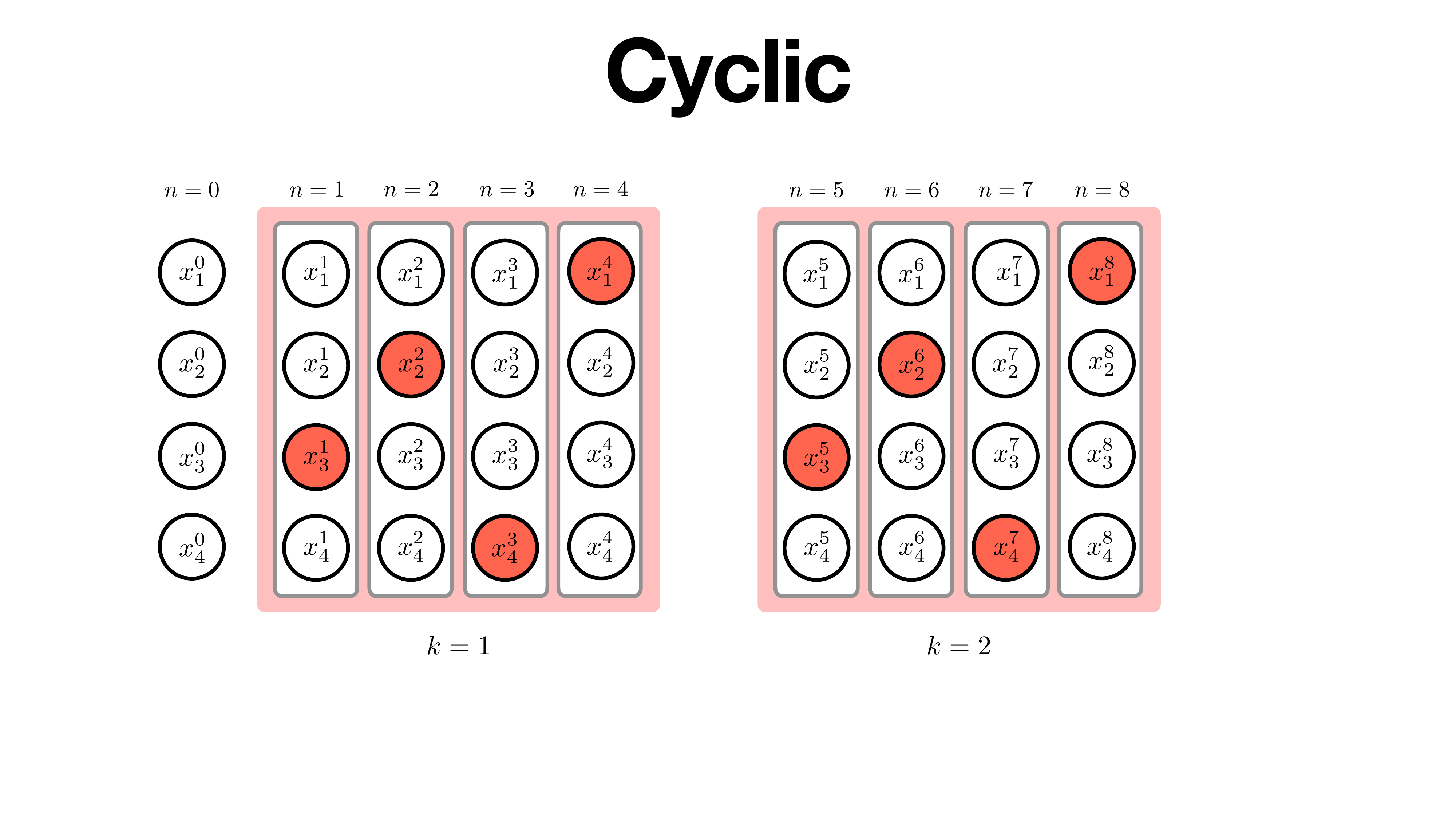}   & \includegraphics[trim={21em 16em 25em 20em},clip,width=0.5\textwidth]{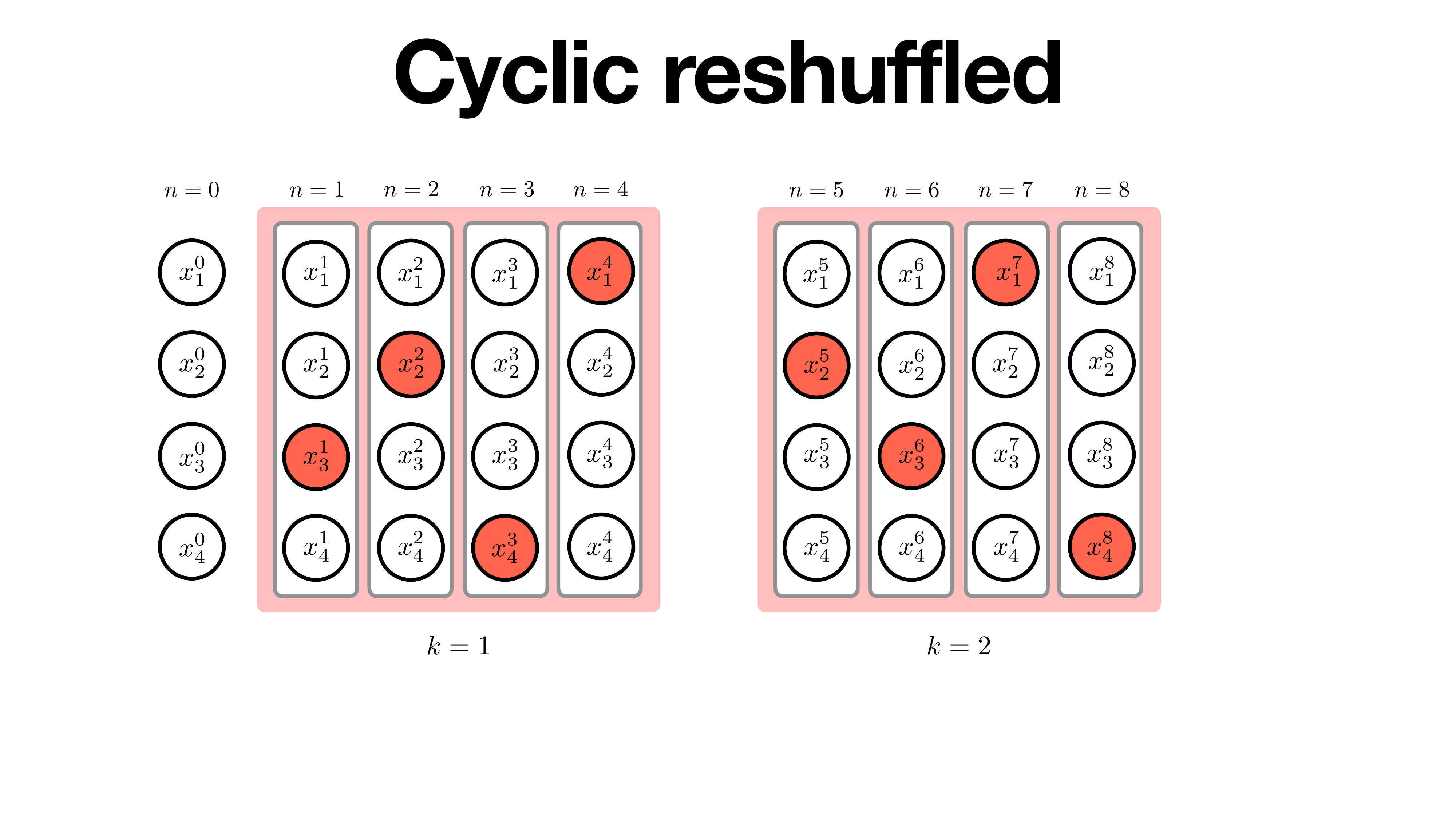} \\
        \end{tabular}
        \caption{Existing update rules}
    \end{subfigure}
    \vspace{2em}
        \begin{subfigure}[b]{\textwidth}
 \renewcommand{\arraystretch}{0.5}
        \setlength{\tabcolsep}{1pt}
        \begin{tabular}{cc}
          Parallel \& essentially cyclic & Hierarchical \\
          \includegraphics[trim={21em 16em 25em 20em},clip,width=0.5\textwidth]{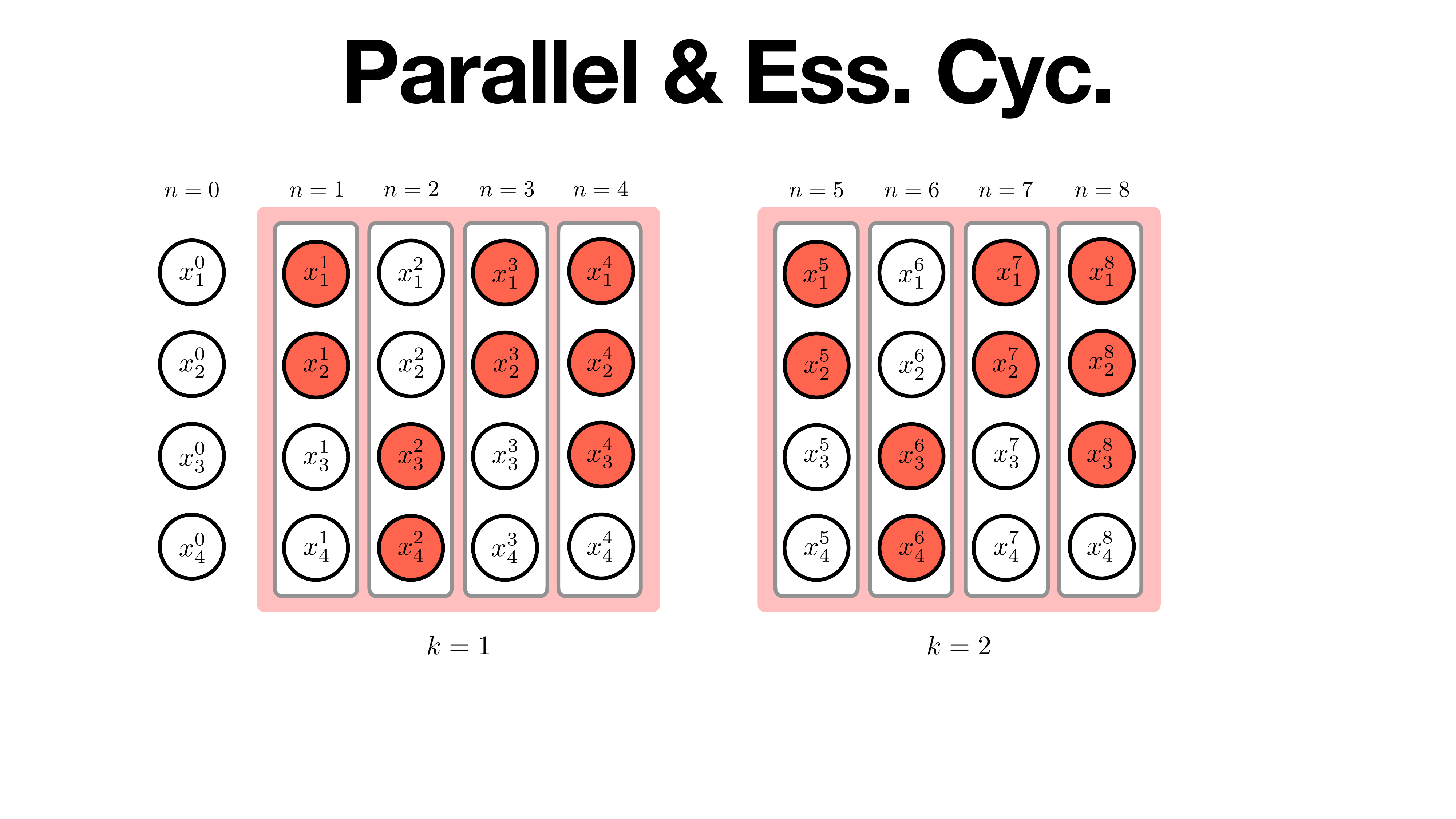}   & \includegraphics[trim={21em 16em 25em 20em},clip,width=0.5\textwidth]{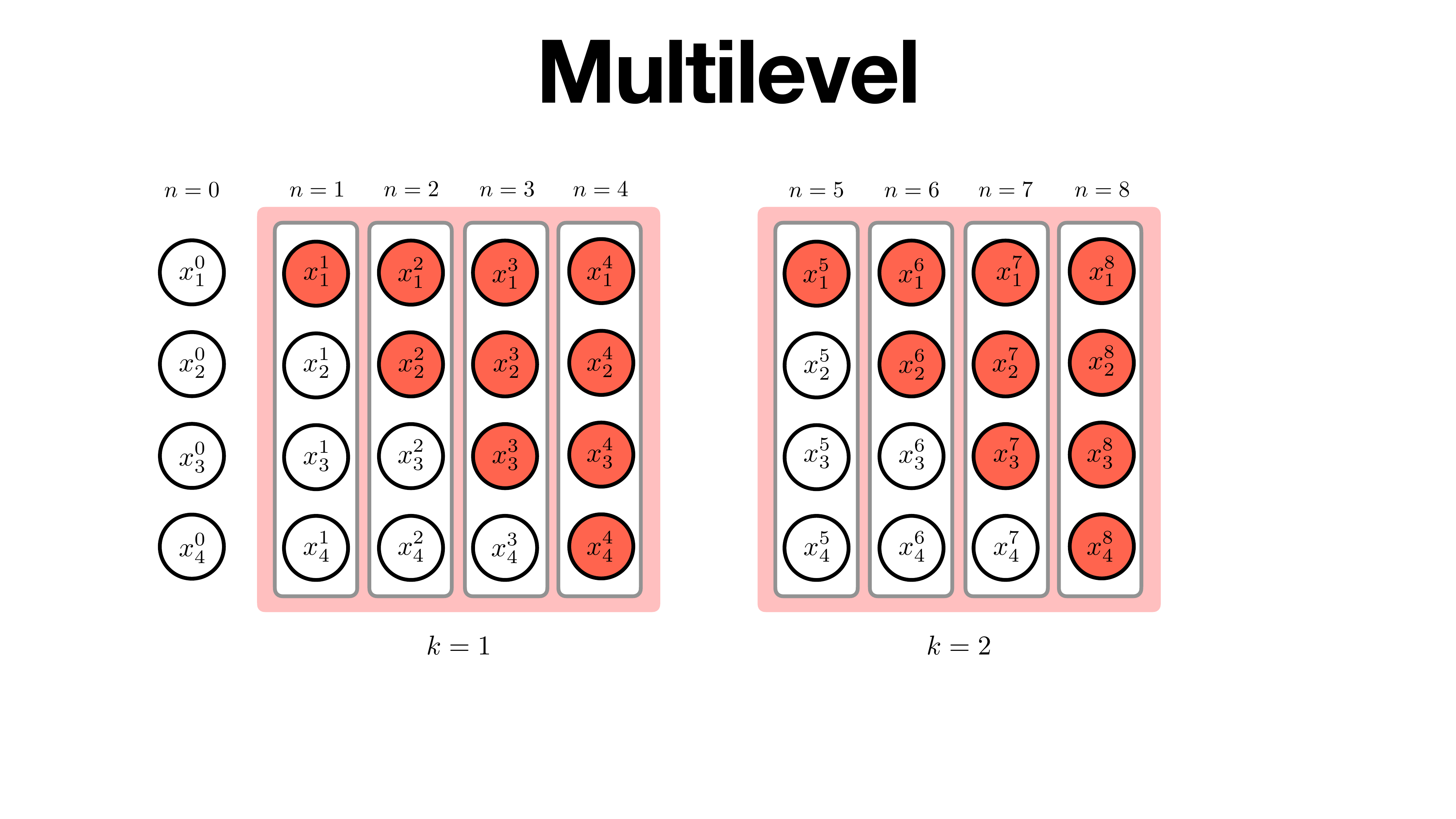} 
        \end{tabular}
        \caption{New update rules}
\end{subfigure}
    \caption{%
    \textbf{Examples of update rules that are covered by our proposed algorithm $\acronym$ for a problem with $4$ blocks}. The top row displays existing rules that are covered by our framework, the bottom displays new rules that are now covered by our framework. %
    The variable of the $\ell$-th block at iteration $n$ is indexed as $x_\ell^n$. At the top of each scheme, we display the iteration number $n$, and at the bottom the cycle number $k$. At each iteration we highlight in red the blocks that are updated. Thus, each column depicts the activation of the blocks at a given iteration. All these rules share a common feature, necessary to make the corresponding BCD method convergent: each block must have been updated at least once during a cycle. Here, each cycle contains $4$ iterations, but it is not necessary for the number of iterations per cycle to be equal to the number of blocks. %
    }
    \label{fig:updates_scheme}
\end{figure*}

Our second contribution is to study the connection between $\acronym$ and multilevel algorithms. %
We show that, for a specific image restoration problem that has a hierarchical structure (cf. \eqref{eq5:wavelet_optim_exp}), the proposed update can encompass a hierarchical  choice of the blocks, mimicking multilevel iterations. 
In particular, we prove that  the first-order coherence, usually imposed in multilevel methods, is actually essential for establishing the equivalence of the two methods. As a result, just like multilevel methods, the $\acronym$ algorithm we propose is faster than the classical forward-backward (FB) in practice. 
Moreover, thanks to this connection, the theory presented in this article can be seen as an extension of the convergence theory of multilevel FB \cite{lauga_iml_2023} to a non-convex setting. %

\paragraph{Outline.} 
This article is organized as follows. In Section \ref{sec:BCD}, we present the state-of-the art BCD strategies to tackle Problem \eqref{eq5:optim} and the proposed $\acronym$ algorithm. In Section \ref{sec5:convergence_result}, we prove the convergence of $\acronym$ in a non-smooth and non-convex setting, where we assume that $\Psi$ satisfies the Kurdyka-Łojasiewicz inequality.  %
In Section \ref{sec5:motivating_example} we show, for a particular instance of problem \eqref{eq5:optim}, that a multilevel forward-backward algorithm can be seen as an instance of the $\acronym$ algorithm with a hierarchical update rule. Finally in Section \ref{sec:numerical}, we compare several instances of our $\acronym$ algorithm to BC-PG strategies encountered in the literature.

\medskip

\paragraph{Notations.} We introduce the notations that will be used in the following. We use bold letters to consider full vectors of variables ($\mathbf{x}$) and plain letters indexed by $\ell$ to denote blocks ($x_\ell$). 
We denote by $\vertiii{\cdot}$ the Euclidean norm on $\Hi$ and by $\Vert \cdot \Vert$ the Euclidean norm on the $L$ spaces $(\Hi_\ell)_{1\leq \ell \leq L}$. Similarly, the scalar product on $\Hi$ will be denoted by $\langle \langle \cdot,\cdot \rangle \rangle$ and the scalar product on $\Hi_\ell$ by  $\langle \cdot,\cdot \rangle$; the potential ambiguity between two spaces is cleared up as the variables on which the scalar product is applied will be indexed by $\ell$. Note that for all $\mathbf{x}$ and $\mathbf{y}$ in $\Hi$, $\langle \langle \mathbf{x},\mathbf{y} \rangle \rangle = \sum_{\ell=1}^L \langle x_\ell,y_\ell \rangle$. For a continuously differentiable function $f$, $\nabla_\ell f(\mathbf{x})$ is the gradient of $f$ with respect to the variables in the $\ell$-th block at some $\mathbf{x} \in \Hi$.  Recall that, for a function $g$ that is proper and lower semicontinuous with $\inf_\Hi g>-\infty$, the proximity operator of $g$ of parameter $\tau >0$ is defined as 
\begin{equation}\label{eq:prox_equation}
    (\forall \mathbf{x} \in \Hi), \quad \prox_{\tau g}(\mathbf{x}) = \Argmin_{\mathbf{u} \in \Hi} \frac{1}{2\tau} \Vert \mathbf{u}-\mathbf{x} \Vert^2 + g(\mathbf{u}).
\end{equation}
where $\Argmin_{u \in \Hi} h(\mathbf{u})$ denotes the set of minimizers of some function $h : \Hi \to (-\infty,+\infty]$. Given $(\mathbf{x},\mathbf{u}) \in \Hi \times \Hi$, we have \cite{bolte2014proximal}
\begin{equation}\label{eq:subgradient_prox}
    \mathbf{u} \in \prox_g(\mathbf{x}) \implies \mathbf x - \mathbf u \in \partial g(\mathbf u).
\end{equation}
We will say that a function $g$ is proximable when $\prox_{\tau g}$ is known under closed form.
\section{Block-coordinate forward-backward algorithm} \label{sec:BCD}
The idea of splitting an optimization problem into smaller tasks is ubiquitous in practice and has sparked in the last twenty years a lot of research to better understand its potential from a theoretical perspective. The following paragraphs describe the bulk of these studies in the context of the BC-PG algorithm, where block updates are done using proximal-gradient descent to handle the non-smoothness of $g:=\sum_{\ell=1}^Lg_\ell$. %
A complete overview of the update methods for BC-PG algorithm may be found in \cite{wright2015coordinate,nutini2022let}. %
\subsection{State-of-the-art strategies}
The most standard formulation of block-coordinate forward-backward algorithm is the following.   We index the  sequence of iterates by a superscript denoting the iteration number and a subscript denoting the block of variables. Thus, $x_\ell^n$ denotes the $\ell$-th block at the $n$-th iteration. For convenience, we write $\mathbf{x}^n$ to denote the full variable at iteration $n$, so that $\mathbf{x}^n = (x_1^{n},\ldots,x_L^{n})$. We denote $(\boldsymbol{\varepsilon}^n)_{n\in \mathbb{N}} = (\varepsilon_1^{n},\ldots,\varepsilon_L^{n})_{n\in \mathbb{N}}$ a sequence of variables with value in $\{0,1\}^L$ and the step-sizes $(\tau_\ell^n)_{1\leq \ell \leq L} \in \RR^{L}_{++}$, for all $n \in \mathbb{N}$. The algorithm is initialized with $\mathbf{x}^0 = (x_1^{0},\ldots,x_L^{0}) \in $ dom $g$ and reads  
\begin{equation}\label{eq5:BC_FB}
\begin{array}{l}
\text{for } n=0,1,\dots \\ 
\left\lfloor
\begin{array}{l}
\text{for } \ell=1,\ldots,L \\
\left\lfloor 
\begin{array}{l}
    x_\ell^{n+1} \in x_\ell^n + \varepsilon_\ell^n \left( \prox_{\tau_\ell^n g_\ell} \left(x_\ell^n - \tau_\ell^n \nabla_\ell f(\mathbf{x}^n) \right) - x_\ell^n\right).
\end{array}
\right.
\end{array}
\right.
\end{array}
\end{equation}
\medskip
\noindent The settings investigated in the literature are the following ones:
\begin{enumerate}
    \item Stochastic setting:  $(\varepsilon_1^{n},\ldots,\varepsilon_L^{n}) \in \{0,1\}^L$ are chosen randomly, enabling random parallel updates, for all $n\in\mathbb{N}$. 
    \item Essentially cyclic setting: given $K>0$,
    \begin{equation*}
            (\forall j \in \mathbb{N}), \quad \bigcup_{n=j}^{j+K-1} \{\ell\mid \varepsilon_\ell^j=1\} = \{1,\ldots,L\}.
    \end{equation*}
\end{enumerate}
There have been numerous works to study the algorithm in \eqref{eq5:BC_FB} 
in the stochastic setting \cite{salzo_parallel_2022,briceno2022random,combettes2015stochastic,richtarik2014iteration,richtarik2016parallel,fercoq2015,wright2015coordinate,cadoni2016block,namkoong2017adaptive,lee2019random,sun2021worst}, in the essentially cyclic one \cite{ortega2000iterative,tseng2001convergence,wright2015coordinate,nutini2015coordinate,nesterov2012efficiency,chouzenoux2016block,bolte2014proximal,abboud2019alternating,xu2017globally} and finally using greedy rules that are efficient in practice but lack the guarantees of the previous two.
This list of references is not exhaustive, but it is representative of the proof techniques used to study the convergence of algorithm \eqref{eq5:BC_FB}.

\subsection{Proposed setting} %

As highlighted in the introduction, despite the huge literature on block coordinate methods, no algorithm allowing a deterministic control of parallel updates with convergence guarantees has yet to be proposed. 

We are also motivated to draw connections %
between multilevel algorithms and BCD algorithms, and the method we propose can benefit in the relevant setting (see Section \ref{sec:numerical}) from the efficiency of multilevel algorithms and the convergence guarantees of the BCD formalism.
Specifically, we design a convergent block-coordinate descent algorithm for non-smooth and non-convex optimization where the updates are potentially parallel, essentially cyclic, and may be randomly reshuffled at each cycle. %
Our framework encompasses, for instance, BCD algorithms where the size of the blocks may vary from one iteration to the other \cite{nutini2022let} or hierarchical BCD, whose update rule mimics the behavior of multilevel algorithms (i.e., some blocks, deemed more important, are updated more often than others, see Figure \ref{fig:updates_scheme} or Section \ref{sec5:motivating_example}). This requires to allow correlated block selection, which is not permitted under the classical convergent stochastic framework \cite{combettes2015stochastic}. Our major contribution is therefore the convergence of the method both in function values and with respect to the set of critical points.
\medskip

\paragraph{$\acronym$ iterations}
Our update rule is \emph{parallel and essentially cyclic} by setting \textit{a priori} the sequence $\boldsymbol{\varepsilon}^n = (\varepsilon_1^{n},\ldots,\varepsilon_L^{n})\in\{0,1\}^L$ for all $n\in \mathbb{N}$. %  
To simplify the following analysis, we rewrite algorithm \eqref{eq5:BC_FB} to explicitly incorporate  the cycles. Let $K\in\mathbb{N}\backslash\{0\}$ be the number of iterations to complete one cycle. % 
Let $(\tau_\ell^n)_{1\leq \ell \leq L} \in \RR^{L}_{++}$ for all $n\in \mathbb{N}$, and $\mathbf{\bar{x}}^0=\mathbf{x}^0 = (x_1^{0},\ldots,x_L^{0}) \in $ dom $g$. Set $k=0$. The iterations read% 
\begin{equation}\label{eq5:H_BC_FB}
\begin{array}{l}
\text{for } n=0,1,\dots \\ 
\left\lfloor
\begin{array}{l}
\text{for } \ell=1,\ldots,L \\
\left\lfloor 
\begin{array}{l}
    x_\ell^{n+1} = x_\ell^n + \varepsilon_\ell^n \left( \prox_{\tau_\ell^n g_\ell} \left(x_\ell^n - \tau_\ell^n \nabla_\ell f(\mathbf{x}^n) \right) - x_\ell^n\right).
\end{array}
\right. \\
\text{if } n+1 \equiv 0~[K]\\
\left\lfloor 
\begin{array}{l}
    k = k+1 \\
    \mathbf{\bar{x}}^{k} = \mathbf{x}^{n+1}
\end{array}
\right.
\end{array} 
\right.
\end{array}
\end{equation}

The convergence analysis relies on a cyclic rule for the updates, % 
and we assume that each cycle consists of at most $K$ iterations. % 
We will denote $\mathbf{\bar{x}}^k$ the iterates obtained after $k$ cycles and thus $k\cdot K$ iterations to accentuate the difference with the iterates $\mathbf{x}^n$. We will prove the convergence of $(\mathbf{\bar{x}}^k)_{k\in\mathbb{N}}$ to a critical point, which will give us, as a byproduct, the convergence of $\mathbf{x}^n$ to the same critical point. Here, $\mathbf{\bar{x}}^k = \mathbf{x}^{n}$ when $n=k \cdot K$. These notations are illustrated in Figure \ref{fig:updates_scheme}.% 

\section{Convergence of the proposed BC-PG algorithm} \label{sec5:convergence_result}

We now analyse the convergence of the proposed scheme using the Kurdyka-Łojasiewicz property of $\Psi$, which allows us to provide both convergence of the objective function values and of the iterates to a critical point of $\Psi$. 
For completeness of the argument, we prove the convergence of a stochastic version of our $\acronym$  algorithm in Appendix \ref{secA:stochastic}.%

\subsection{Preliminaries in non-convex optimization}

In the non-convex setting, we will need an appropriate notion of subgradient. For the rest of the paper, we assume that $\Hi := \RR^N$. \medskip
\begin{definition}{\textbf{Subdifferential \cite{VarAnalRockafellar}.}}
    Let $\Psi:\Hi \to \RR$, and let $\mathbf{x}\in\Hi$. The Fréchet subdifferential of $\Psi$ at $\mathbf{x}$ is denoted by $\hat{\partial} \Psi(\mathbf{x})$ and is given by the set
    \begin{align*}
        \hat{\partial} \Psi(\mathbf{x}) = \left\{ \hat{\mathbf{s}} \in \Hi \;|\; \lim_{\mathbf{y}\rightarrow \mathbf{x}} \inf_{\mathbf{y} \neq \mathbf{x}} \frac{1}{\vertiii{\mathbf{x}-\mathbf{y}}} \left(\Psi(\mathbf{y})-\Psi(\mathbf{x})-\langle\langle \mathbf{y}-\mathbf{x}, \hat{\mathbf{s}} \rangle\rangle \right) \geq 0 \right\}.
    \end{align*}
    If $\mathbf{x} \notin \textrm{dom } \Psi$, then $\hat{\partial} \Psi(\mathbf{x}) = \emptyset$.
    The limiting subdifferential of $\Psi$ at $\mathbf{x}$ is denoted by $\partial \Psi(\mathbf{x})$ and is given by 
    \begin{align*}
\partial \Psi(\mathbf{x}) = \big\{ \mathbf{s} \in \Hi \;|\; \exists & \left( \mathbf{x}^k,\hat{\mathbf{s}}^k \right) \overset{k\rightarrow\infty}{\rightarrow} \left(\mathbf{x},\mathbf{s}\right) \\ & \textrm{ such that } \Psi(\mathbf{x}^k) \overset{k\rightarrow\infty}{\rightarrow} \Psi(\mathbf{x}) \textrm{ and } (\forall k \in \mathbb{N}) ~\hat{\mathbf{s}}^k \in \hat{\partial} \Psi(\mathbf{x}^k) \big\}.
    \end{align*}
Recall that if $\Psi$ is convex, its subdifferential is given, for all $\mathbf{x} \in \Hi$, by
\begin{equation*}
%\label{eq5:subdif_convex}
    \partial \Psi(\mathbf{x}) = \{\mathbf{s} \in \Hi \;|\; \Psi(\mathbf{x}) + \langle\langle \mathbf{s}, \mathbf{y}-\mathbf{x} \rangle\rangle \leq \Psi(\mathbf{y}), \forall \mathbf{y} \in \Hi \}.
\end{equation*}
Both $\hat{\partial} \Psi(x)$ and $\partial \Psi(x)$ are closed sets \cite[Theorem 8.6]{VarAnalRockafellar}.
\end{definition}

\vspace{0.3cm}
The limiting subdifferential benefits from the following separability property:
\begin{proposition}\label{prop:sub}{\textbf{Subdifferentiability property \cite{VarAnalRockafellar}.}} Let $\Psi$ be defined as in  problem \eqref{eq5:optim}. Then, for all $x = (x_1,\ldots,x_L)\in \Hi_1 \times \ldots \times\Hi_L$, we have
    \begin{equation*}
        \partial \Psi(\mathbf{x}) = (\nabla_1 f(\mathbf{x}) + \partial g_1(x_1)) \times \ldots \times (\nabla_L f(\mathbf{x}) +\partial g_L(x_L)).
    \end{equation*}
    \end{proposition} \medskip
\vspace{-0.3cm}
\paragraph{The Kurdyka-Łojasiewicz (KŁ) property.}
A specific class of concave and continuous functions, called desingularizing functions, are of particular interest in the KŁ framework to handle non-convexity. \medskip
\begin{definition}{\textbf{Concave and continuous functions \cite{bolte2014proximal}.}} \label{def:desingularizing} Let $\eta \in (0,+\infty]$. We denote by $\Phi_\eta$ the class of all concave and continuous functions $\varphi : [0,\eta)\rightarrow\RR_+$ that satisfy the following conditions:
\begin{enumerate}
    \item $\varphi(0)=0$,
    \item $\varphi$ is $C^1$ on $(0,\eta)$ and continuous at $0$,
    \item for all $s \in (0,\eta)$, $\varphi'(s)>0$.
\end{enumerate}
\end{definition}
Now, we can introduce the definition of a KŁ function. \medskip
\begin{definition}{\textbf{Kurdyka-Łojasiewicz (KŁ) property \cite{bolte2014proximal}.}} \label{def:kl}
    Let $\Psi : \Hi \rightarrow (-\infty, + \infty]$ be proper and lower semicontinuous.
    \begin{enumerate}
        \item The function $\Psi$ is said to have the \textit{KŁ property} at $\mathbf{\tilde{u}}\in$ dom $\partial\Psi$ if there exist $\eta \in (0,+\infty]$, a neighborhood $U$ of $\mathbf{\tilde{u}}$ and a function $\varphi \in \Phi_\eta$  such that for all 
        \begin{equation*}
            \mathbf u \in U \cap [\Psi(\tilde{\mathbf u}) < \Psi(\mathbf{u}) < \Psi(\tilde{\mathbf u}) + \eta],
        \end{equation*}
        the following inequality holds
        \begin{equation*}
            \varphi'(\Psi(\mathbf{u}) - \Psi(\mathbf{\tilde{u}})) \textrm{dist} (0,\partial \Psi(\mathbf{u})) \geq 1.
        \end{equation*}
        Recall that $\textrm{dist}(\mathbf{x},\partial \psi (\mathbf{u})) = \inf_{\mathbf{s} \in \partial \psi(\mathbf{u})} \vertiii{\mathbf s- \mathbf{x}}$.
        \item If $\Psi$ satisfies the KŁ property at each point of dom $\partial\Psi$, then $\Psi$ is called a KŁ function.
    \end{enumerate}
\end{definition}

The following lemma presents the KŁ property in a practical form by unifying the notion of neighborhood across its level curves. \medskip
\begin{lemma}{\textbf{Uniformized KŁ property \cite{bolte2014proximal}.}} \label{lm:unif_KL}Let $\Omega$ be a compact subset of $\Hi$. Let $\Psi:\Hi \to (-\infty,+\infty]$ be a proper and lower semicontinuous function, constant on $\Omega$ and satisfying the KŁ inequality on $\Omega$. Then there exists $\epsilon>0, \eta>0$, and $\varphi\in\Phi_\eta$ such that for all $\tilde{u}\in \Omega$ and all $u$ satisfying
\begin{equation*}
    \left\{\mathbf u \in \Hi: \textrm{dist} (u,\Omega) < \epsilon \right\} \text{ and } [\Psi(\tilde{\mathbf u}) < \Psi(\mathbf{u}) < \Psi(\tilde{\mathbf u}) + \eta]
\end{equation*}
one has
\begin{equation}
    \varphi'(\Psi(\mathbf{u}) - \Psi(\tilde{\mathbf u})) \textrm{dist} (0, \partial \Psi(\mathbf{u})) \geq 1. \label{eq5:inequality_KL}
\end{equation}
\end{lemma}
\begin{remark}
The KŁ property is satisfied by numerous classes of functions, and notably by those considered in typical optimization settings. See \cite{bolte2010characterizations} for an overview on this property. 
\end{remark}

\subsection{Assumptions on the functions.}
The convergence of $\acronym$ relies on several classical assumptions that we present in the following. \medskip
\begin{assumption} \label{ass:deterministic1}
\textcolor{white}{y}
    \begin{itemize}
        \item[\normalfont A$1$] $\Psi:=f+g$, where $g:=\sum_{\ell=1}^Lg_\ell$, is coercive, and bounded below. For all $\ell\in \{1,\ldots,L\}$, $g_\ell$ and $f$ are bounded from below.
        \item[\normalfont A$2$] $\Psi$ satisfy the KŁ property (Definition \ref{def:kl} and Lemma \ref{lm:unif_KL}).
    \end{itemize}
\end{assumption} \medskip

\begin{assumption} \label{ass:deterministic2} 
\textcolor{white}{y}
    \begin{itemize}
        \item[\normalfont A$3$] For all $\ell \in \{1,\ldots,L\}$, $g_\ell$ is a lower semicontinuous, proper function.
        \item[\normalfont A$4$] $f$ is continuously differentiable and there exist constants $(\beta_{\ell,j})_{1 \leq \ell,j \leq L}$ in $\RR_{++}$ such that
        \begin{align*}
            (\forall \ell,j\in\{1,\ldots,L\})&(\forall \mathbf{x} \in \Hi)(\forall v_j \in \Hi_j) \nonumber\\
            & \Vert \nabla_\ell f(\mathbf{x}+(0,\ldots,0,v_j,0,\ldots,0)) - \nabla_\ell f(\mathbf{x}) \Vert \leq \beta_{\ell,j} \Vert v_j \Vert 
        \end{align*}
        
    \end{itemize}
\end{assumption}
%\begin{remark}
Assumption {\normalfont A$1$} is sufficient to assert that the sequences generated by our algorithm are bounded \cite{attouch2010proximal}. 

Assumption {\normalfont A$4$} states that every partial gradient with respect to the block is Lipschitz continuous with respect to all the blocks, which is a quite stronger assumption than being Lipschitz continuous with respect only to its block. From this assumption we can derive multiple block smoothness, a common assumption in the BCD literature (e.g., \cite[Assumption S1-S2-S3]{salzo_parallel_2022}).
Despite this, Assumption {\normalfont A$4 $} is fairly easy to verify in practice, since it is implied by the Lipschitz continuity of $\nabla f$ with constant $\beta_f$. Indeed, we can take $\beta_{\ell,j} = \beta_f$ for all $\ell,j$. Conversely, Assumption {\normalfont A$4$} implies that $\nabla f$ is at most $\beta_f := \sqrt{\sum_{\ell,j=1}^L \beta_{\ell,j}^2}$-Lipschitz continuous. This is a consequence of the following proposition. \medskip
\begin{proposition}{\textbf{Multiple block smoothness.}} \label{prop:multi_block_smoothness}
    Suppose that Assumption \ref{ass:deterministic2} holds. For all $\boldsymbol{\varepsilon} = (\varepsilon_\ell)_{1 \leq \ell \leq L} \in \{0,1\}^L$, there exists $\beta=\sqrt{\sum_{\ell,j=1}^L \varepsilon_j \beta_{\ell,j}^2}>0$ such that for all $\mathbf{v} = (v_\ell)_{1 \leq \ell \leq L} \in (\Hi_1,\dots,\Hi_L)$ we have
    \begin{equation*}
        (\forall \mathbf{x} \in \Hi) \quad \vertiii{\nabla f(\mathbf{x}+\boldsymbol{\varepsilon}\odot \mathbf{v})-\nabla f ( \mathbf{x})  } \leq \beta \vertiii{\boldsymbol{\varepsilon}\odot \mathbf{v}},
    \end{equation*}
   where $\odot$ is the element-wise multiplication $\boldsymbol{\varepsilon} \odot \mathbf{v} = (\varepsilon_1 v_1, \varepsilon_2 v_2,\ldots, \varepsilon_L v_L)$.
\end{proposition}
\begin{proof}
   Let $\mathbf{x}, \mathbf{v}=(v_{\ell})_{\ell=1}^L\in\Hi=\oplus_{\ell=1}^L\Hi_{\ell}$. Note that 
\begin{equation*}
|||\nabla f(\mathbf{x}+\mathbf{v})-\nabla f(\mathbf{x})|||^2=\sum_{\ell=1}^L\|\nabla_{\ell}f(\mathbf{x}+\mathbf{v})-\nabla_{\ell}f(\mathbf{x})\|^2.
\end{equation*}
Now define,  %
$\mathbf{v}_0=\mathbf{0}\in\Hi$ and   $\mathbf{v}_j=(v_1,\ldots,v_j,0,\ldots,0)\in\Hi$ for every $j\in\{1,\ldots,L\}$. Note that $\mathbf{v}_{L}=\mathbf{v}$ and that
\begin{equation*}
(\forall j\in\{1,\ldots,L\})\quad 
\mathbf{v}_j-\mathbf{v}_{j-1}=(0,\ldots, v_{j},\ldots,0).
\end{equation*}
For every $\ell\in\{1,\ldots,L\}$, the following equality holds
\begin{align}
\|\nabla_{\ell}f(\mathbf{x}+\boldsymbol{\varepsilon} \odot \mathbf{v})-\nabla_{\ell}f(\mathbf{x})\|=\Big\|\sum_{j=1}^L(\nabla_{\ell}f(\mathbf{x}+\boldsymbol{\varepsilon} \odot \mathbf{v}_{j})-\nabla_{\ell}f(\mathbf{x}+\boldsymbol{\varepsilon} \odot \mathbf{v}_{j-1}))\Big\|\nonumber
\end{align}
Then, the triangular inequality and A4 imply 
\begin{align}
\|\nabla_{\ell}f(\mathbf{x}+\boldsymbol{\varepsilon} \odot \mathbf{v})-\nabla_{\ell}f(\mathbf{x})\| &\leq
\sum_{j=1}^L\|\nabla_{\ell}f(\mathbf{x}+\boldsymbol{\varepsilon} \odot \mathbf{v}_{j})-\nabla_{\ell}f(\mathbf{x}+\boldsymbol{\varepsilon} \odot \mathbf{v}_{j-1})\|\nonumber\\
&\leq \sum_{j=1}^L  \beta_{\ell,j}\|\varepsilon_j v_j\|
\end{align}
and therefore, from Cauchy-Schwarz in $\RR^L$,
\begin{align}
|||\nabla f(\mathbf{x}+\boldsymbol{\varepsilon}\odot\mathbf{v})-\nabla f(\mathbf{x})|||^2\leq \sum_{\ell=1}^L\left(\sum_{j=1}^L \beta_{\ell,j}\|\varepsilon_j v_j\|\right)^2\leq\sum_{\ell,j=1}^L\beta_{\ell,j}^2|||\boldsymbol{\varepsilon} \odot \mathbf{v}|||^2,
\end{align}
deducing that $\beta=\sqrt{\sum_{\ell,j=1}^L \varepsilon_j \beta_{\ell,j}^2}$ is a Lipschitz constant of $\nabla f$ with respect to the blocks selected by $\boldsymbol{\varepsilon}$. %
\end{proof}
\subsection{Assumptions on the update rules.}

We consider an essentially cyclic update scheme for the blocks in which parallel updates of different blocks may be used, paired with a potential shuffle of the updates order, as specified in the following assumption. \medskip
\begin{assumption} \label{ass:deterministic3} 
    \begin{itemize}    
        \item[\normalfont A$5$] Let $\mathbb{I}^n$ be the set of the blocks updated at iteration $n$, i.e., 
        \begin{equation}\label{eq:In_def}
        \mathbb{I}^n =\{\ell \;|\; \varepsilon_\ell^n=1\} \subseteq \{1,\ldots,L\}.
        \end{equation}
        
        There exists $K\in \mathbb{N}\backslash\{0\}$
        such that
        \begin{equation*}
            (\forall j \in \mathbb{N}), \quad \bigcup_{n=j}^{j+K-1} \mathbb{I}^n = \{1,\ldots,L\}. 
        \end{equation*}
    \end{itemize}
\end{assumption}
\begin{remark}
    Assumption {\normalfont A$5$} does not  impose any constraint on the order of the updates of the blocks. For instance, it allows for sequential update of the blocks if $K = L$ and it reduces to the classical forward-backward update for $K=1$. 
    As an example of the flexibility of our framework, one can shuffle the order of the updates inside every cycle without breaking convergence guarantees. Hence, a random shuffle such as in \cite{xu2017globally} is compatible with Assumption {\normalfont A$5$}.
\end{remark}

Before setting the main result, we need one more technical lemma. % 

\begin{lemma} \label{lm:norm_bound}
    Let $\{\mathbf{\bar{x}}^k\}_{k\in\mathbb{N}}$ be a sequence generated by Algorithm $\acronym$. Then,
    \begin{equation*}
        \vertiii{\mathbf{\bar{x}}^{k+1}-\mathbf{\bar{x}}^{k}} \leq \left(\sum_{n=k\cdot K}^{(k+1) \cdot K-1} \sum_{\ell \in \mathbb{I}^n} \Vert x_\ell^{n+1}-x_\ell^n \Vert \right),
    \end{equation*}
    where $\mathbb{I}^n$ is defined in \eqref{eq:In_def}. 
\end{lemma}
\begin{proof}
        Recall that 
        \begin{align*}
            \vertiii{\mathbf{\bar{x}}^{k+1}-\mathbf{\bar{x}}^{k}}  = \sqrt{\sum_{\ell=1}^L \Vert \mathbf{\bar{x}}^{k+1}_\ell-\mathbf{\bar{x}}^{k}_\ell\Vert^2} \leq \sum_{\ell=1}^L \sqrt{\Vert \mathbf{\bar{x}}^{k+1}_\ell-\mathbf{\bar{x}}^{k}_\ell\Vert^2}  =  \sum_{\ell=1}^L \Vert \mathbf{\bar{x}}^{k+1}_\ell-\mathbf{\bar{x}}^{k}_\ell\Vert.
        \end{align*}
        Now for all $\ell\in\{1,\ldots,L\}$, the triangular inequality yields
        \begin{equation*}
            \Vert \mathbf{\bar{x}}^{k+1}_\ell-\mathbf{\bar{x}}^{k}_\ell\Vert \leq \sum_{n=k\cdot K}^{(k+1) \cdot K-1}  \Vert x_\ell^{n+1}-x_\ell^n \Vert. %
        \end{equation*}
        Thus, summing up for all $\ell$, we obtain
        \begin{equation*}    \vertiii{\mathbf{\bar{x}}^{k+1}-\mathbf{\bar{x}}^{k}} \leq \left(\sum_{n=k\cdot K}^{(k+1) \cdot K-1} \sum_{\ell=1}^L \Vert x_\ell^{n+1}-x_\ell^n \Vert \right)= \left(\sum_{n=k\cdot K}^{(k+1) \cdot K-1} \sum_{\ell \in \mathbb{I}^n} \Vert x_\ell^{n+1}-x_\ell^n \Vert \right)
        \end{equation*}
        where the last equality follows from $x_\ell^{n+1} = x_\ell^n$, for all $\ell \notin \mathbb{I}^n$ and $n \in \mathbb{N}$.
    \end{proof}
    
\subsection{Main result: convergence to critical points of $\Psi$}
The convergence to critical points of the sequence $\{\mathbf{\bar x}^k\}_{k \in \mathbb{N}}$ is a consequence of the sufficient decrease property of our algorithm, paired with the existence at each point of a subgradient bounded by the norm of the difference between iterates.
\begin{proposition} \label{prop:decrease}
Suppose that Assumptions \ref{ass:deterministic1}, \ref{ass:deterministic2}, and \ref{ass:deterministic3} hold. Let $\{\mathbf{x}^n\}_{n\in\mathbb{N}}$, and $\{\mathbf{\bar{x}}^k\}_{k\in\mathbb{N}}$ be the sequences generated by Algorithm $\acronym$. The following assertions hold. \medskip %

    \begin{enumerate} 
        \item \emph{Sufficient decrease property}. The sequence $\{\Psi(\mathbf{\bar{x}}^k)\}_{k\in\mathbb{N}}$ is non-increasing. For each $n$, let $\beta_f^n = \sqrt{\sum_{j\in \mathbb{I}^n,1\leq \ell \leq L} \varepsilon_j \beta_{\ell,j}^2}$ and $0<\tau_\ell^n < 1/\beta_f^n$.  Then, for all $k\geq 0$,% 
    \begin{align*}
        \Psi(\mathbf{\bar{x}}^{k+1})+%\rho_k 
        \left(\sum_{n=k\cdot K}^{(k+1) \cdot K-1} \sum_{\ell \in \mathbb{I}^n} \frac{1}{2}\left(\frac{1}{\tau_\ell^n}-\beta_f^n \right)\Vert x_\ell^n-x_\ell^{n+1} \Vert^2 \right) &\leq
        \Psi(\mathbf{\bar{x}}^k).
        \end{align*}
    Furthermore,
    \begin{equation*}
        \sum_{n=0}^{+\infty} \left( \sum_{\ell=1}^{L} \Vert x_\ell^n - x_\ell^{n+1} \Vert^2 \right) < + \infty,
    \end{equation*}
    which implies 
        \noindent $\lim_{n\rightarrow+\infty} \Vert x_\ell^n-x_\ell^{n+1} \Vert =0$ for every $\ell\in\{1,\ldots,L\}$, which in turn yields $\lim_{n\rightarrow+\infty} \vertiii{ \mathbf{x}^n-\mathbf{x}^{n+1}} =0$.
    \item \emph{Subgradient bound}. For each $k \in \mathbb{N}$ define
\begin{equation*}
    \bar{B}^{k+1} = \left(\frac{x_{\ell}^{k_{\ell}}-x_{\ell}^{k_{\ell+1}}}{\tau_{\ell,k_\ell}} - \nabla_\ell f(\mathbf{x}^{k_\ell}) + \nabla_\ell f (\mathbf{\bar{x}}^{k+1})\right)_{1\leq \ell \leq L},
\end{equation*}
where $k_\ell$ is a positive integer such that $k \cdot K \leq k_\ell \leq (k+1) \cdot K -1$, and is the last iteration of cycle $k$ at which block $\ell$ is updated. Then $\bar{B}^{k+1} \in \partial \Psi(\mathbf{\bar{x}}^{k+1})$ and there exist  positive numbers $\tau_k$ such that:%
\begin{equation}\label{eq:bound_subgradient}
    \vertiii{\bar{B}^{k+1}} \leq \left(\frac{1}{\tau_k} + \beta_f\right) \left(\sum_{n=k\cdot K}^{(k+1) \cdot K-1} L \sum_{\ell \in \mathbb{I}^n} \frac{1}{2}\left(\frac{1}{\tau_\ell^n}-\beta_f^n \right)\Vert x_\ell^n-x_\ell^{n+1} \Vert \right).
\end{equation} \end{enumerate} 
\end{proposition}
\begin{proof}
\textbf{1.} 
First, for all $\ell \notin \mathbb{I}^n$, $x_\ell^{n+1} = x_\ell^n$. For all $\ell \in \mathbb{I}^n$, by applying the first order optimality conditions of the proximity operator \eqref{eq:prox_equation} we obtain: %
\begin{equation*}%
        g_\ell(x_\ell^{n+1}) + \frac{1}{2 \tau_\ell^n} \Vert x_\ell^n-x_\ell^{n+1} \Vert^2 \leq g(x_\ell^n) + \langle \nabla_\ell f(\mathbf{x}^{n}), x_\ell^n-x_\ell^{n+1} \rangle,
\end{equation*}
(see Appendix \ref{secA:proofs_BCD}, Lemma \ref{lm:descent_prox}),
which we can sum up for all $\ell \in \mathbb{I}^n$ to obtain
\begin{equation}\label{eq5:subdif_g_Jn}
    \sum_{\ell \in \mathbb{I}^n} \left(g_\ell(x_\ell^{n+1}) + \frac{1}{2 \tau_\ell^n} \Vert x_\ell^n-x_\ell^{n+1} \Vert^2 \right)\leq \sum_{\ell \in \mathbb{I}^n} \big( g(x_\ell^n) + \langle \nabla_\ell f(\mathbf{x}^{n}), x_\ell^n-x_\ell^{n+1} \rangle \big).
\end{equation}
We now invoke {\normalfont A$4$} from Assumption \ref{ass:deterministic2} and by splitting the scalar product along the blocks we get %
\begin{align*}
    f(\mathbf{x}^n + [x_1^{n+1}-x_1^{n},\ldots,x_L^{n+1}-x_L^{n}]^\top) \leq f(\mathbf{x}^n) + \sum_{\ell\in \mathbb{I}^n} & \Big(\langle \nabla_\ell f(\mathbf{x}^n) , x_\ell^{n+1}-x_\ell^n \rangle \\& + \frac{\beta_f^n}{2}\Vert x_\ell^n-x_\ell^{n+1} \Vert^2\Big).
\end{align*}
Note that $f(\mathbf{x}^n + [x_{1}^{n+1}-x_{1}^n,\ldots,x_{L}^{n+1}-x_{L}^n]^\top) = f(\mathbf{x}^{n+1})$, and thus
\begin{equation}
\label{eq5:descent_f_Jn}
    \sum_{\ell \in \mathbb{I}^n} \left(\langle \nabla_\ell f(\mathbf{x}^{n}), x_\ell^n-x_\ell^{n+1} \rangle \right) \leq f(\mathbf{x}^{n})-f(\mathbf{x}^{n+1}) + \frac{\beta_f^n}{2} \sum_{\ell \in \mathbb{I}^n} \Vert x_\ell^n-x_\ell^{n+1} \Vert^2.
\end{equation}
Combining inequalities \eqref{eq5:subdif_g_Jn} and \eqref{eq5:descent_f_Jn} we obtain
\begin{align*}
    \sum_{\ell \in \mathbb{I}^n} \left(g_\ell(x_\ell^{n+1}) + \frac{1}{2\tau_\ell^n} \Vert x_\ell^n-x_\ell^{n+1} \Vert^2\right)\leq \sum_{\ell \in \mathbb{I}^n} & \Big( g(x_\ell^n) + f(\mathbf{x}^{n})-f(\mathbf{x}^{n+1}) \\&+ \frac{\beta_f^n}{2} \sum_{\ell \in \mathbb{I}^n} \Vert x_\ell^n-x_\ell^{n+1} \Vert^2\Big).
\end{align*}
We add $\sum_{\ell \notin \mathbb{I}^n} g_\ell(x_\ell^n)$ to each side of this inequality, and since $\sum_{\ell \notin \mathbb{I}^n} g_\ell(x_\ell^n)=\sum_{\ell \notin \mathbb{I}^n} g_\ell(x_\ell^{n+1})$ (these blocks are not updated),  we have
\begin{equation}
    \Psi(\mathbf{x}^{n+1}) +\sum_{\ell \in \mathbb{I}^n} \frac{1}{2}\left(\frac{1}{\tau_\ell^n}-\beta_f^n \right)\Vert x_\ell^n-x_\ell^{n+1} \Vert^2 \leq \Psi(\mathbf{x}^n). \label{eq5:suff_dec_n}
\end{equation}
Now bundle $K$ block iterations together to define the sequence $(\mathbf{\bar{x}}^k)_{k\in\mathbb{N}}$. Let $k \in \mathbb{N}$, summing inequality \eqref{eq5:suff_dec_n} on all iterations from $n=k \cdot K$ to $n=(k+1) \cdot K-1$, and taking into account that $\mathbf{x}^{k\cdot K}=\mathbf{\bar{x}}^{k}$ and $\mathbf{x}^{(k+1)\cdot K}=\mathbf{\bar{x}}^{k+1}$, we obtain %
    \begin{align*}
        \Psi(\mathbf{\bar{x}}^{k+1}) + \sum_{n=k\cdot K}^{(k+1) \cdot K-1} \sum_{\ell \in \mathbb{I}^n} \frac{1}{2}\left(\frac{1}{\tau_\ell^n}-\beta_f^n \right)\Vert x_\ell^n-x_\ell^{n+1} \Vert^2 \leq \Psi(\mathbf{\bar{x}}^{k}),
    \end{align*}
    and the result follows.\\

\noindent Let us now take $n_0\in\mathbb{N}\backslash\{0\}$. Summing up inequality \eqref{eq5:suff_dec_n} from $n=0$ to $n_0-1$, we obtain
\begin{align*}
    \sum_{n=0}^{n_0-1} \left(\sum_{\ell \in \mathbb{I}^n} \frac{1}{2}\left(\frac{1}{\tau_\ell^n}-\beta_f^n \right)\Vert x_\ell^n-x_\ell^{n+1} \Vert^2 \right) & \leq \Psi(\mathbf{x}^0)-\Psi(\mathbf{x}^{n_0}) \\
    & \leq \Psi(\mathbf{x}^0)- \inf \Psi
\end{align*}
and then, by setting  $C:=\displaystyle{\min_{n=0,\dots,n_0-1}\frac{1}{2}\left(\frac{1}{\tau_\ell^n}-\beta_f^n \right)>0}$ we get 
\begin{equation*}
    \sum_{n=0}^{n_0-1} \left(\sum_{\ell \in \mathbb{I}^n} \Vert x_\ell^n-x_\ell^{n+1} \Vert^2 \right) \leq \frac{1}{C}\left(\Psi(\mathbf{x}^0)-\inf \Psi \right) < + \infty.
\end{equation*}
The result folows by taking $n_0\to+\infty$. %

Now, set
\begin{equation*}
    (\forall k \in \mathbb{N}), \quad \rho_k = \min_{k \cdot K \leq n \leq (k+1) \cdot K-1} \left(\frac{1}{\tau_\ell^n}-\frac{\beta_f^n}{2} \right) \geq 0,
\end{equation*}
and set
\begin{equation}
    C_k = \sum_{n=k\cdot K}^{(k+1) \cdot K-1} \sum_{\ell \in \mathbb{I}^n}  \Vert x_\ell^n-x_\ell^{n+1} \Vert^2. \label{eq:Ck}
\end{equation}

\textbf{2.} For all $k\in\mathbb{N}$ and for all $\ell\in\{1,\ldots,L\}$, there exists an iteration index $k_\ell \in \mathbb{N} \cap [k\cdot K, (k+1) \cdot K-1]$ in which block $\ell$ receives its last update in cycle $k$. Thus for all $\ell\in \{1,\ldots,L\}$ from \eqref{eq:subgradient_prox}
\begin{equation*}
    \frac{x_{\ell}^{k_{\ell}}-x_{\ell}^{k_{\ell+1}}}{\tau_{\ell}^{k_\ell}} - \nabla_\ell f(\mathbf{x}^{k_\ell}) \in \partial g_\ell (x_{\ell}^{k_{\ell+1}}),
\end{equation*} 
which yields
\begin{align*}
    \bar{B}_\ell^{k+1} & := \frac{x_{\ell}^{k_{\ell}}-x_{\ell}^{k_{\ell+1}}}{\tau_{\ell}^{k_\ell}} - \nabla_\ell f(\mathbf{x}^{k_\ell}) + \nabla_\ell f (\mathbf{\bar{x}}^{k+1})\\ & \in \partial g_\ell (x_{\ell}^{k_{\ell+1}}) + \nabla_\ell f (\mathbf{\bar{x}}^{k+1}) 
 = \partial_\ell \Psi(\mathbf{\bar{x}}^{k+1}).
\end{align*}
where the last equality follows from Proposition \ref{prop:sub}. %
Hence, by setting $\bar{B}^{k+1} = (\bar{B}^{k+1}_\ell)_{1 \leq \ell \leq L}$, we have $\bar{B}^{k+1} \in \partial \Psi(\mathbf{\bar x}^{k+1}$. An upper bound on the norm of $\bar{B}^{k+1}$ follows from the block-smoothness property of $f$. We have for all $\ell\in\{1,\ldots,L\}$
\begin{align*}
    \Vert \nabla_\ell f(\mathbf{\bar{x}}^{k+1})-\nabla_\ell f(\mathbf{x}^{k_\ell}) \Vert & \leq \sum_{j=1}^L \beta_{\ell,j} \Vert \bar{x}^{k+1}_j - x^{k_\ell}_j \Vert \nonumber \\
    & = \sum_{j=1}^L \beta_{\ell,j} \left\Vert \sum_{n=k_\ell}^{(k+1)\cdot K-1} x_j^{n+1} - x_j^n \right\Vert \nonumber \\
    & \leq \sum_{j=1}^L \beta_{\ell,j} \sum_{n=k_\ell}^{(k+1)\cdot K-1} \Vert x_j^{n+1} - x_j^n \Vert \nonumber \\
    & \leq \left(\max_{1\leq \ell,j \leq L} \beta_{\ell,j} \right)\sum_{j=1}^L  \sum_{n=k\cdot K}^{(k+1)\cdot K-1} \Vert x_j^{n+1} - x_j^n \Vert \nonumber \\
    & \leq \left(\max_{1\leq \ell,j \leq L} \beta_{\ell,j} \right) \sum_{n=k\cdot K}^{(k+1)\cdot K-1} \sum_{j\in \mathbb{I}^n} \Vert x_j^{n+1} - x_j^n \Vert.
\end{align*}
Now define 
\begin{equation*}
    \nu_k = L\left(\frac{1}{\min_{k\cdot K \leq n \leq (k+1) \cdot K-1} \tau_\ell^n} + \max_{1\leq \ell,j \leq L} \beta_{\ell,j}\right), 
\end{equation*}
and 
\begin{equation*}
    D_k = \sum_{n=k\cdot K}^{(k+1)\cdot K-1} \sum_{j\in \mathbb{I}^n} \Vert x_j^{n+1} - x_j^n \Vert.
\end{equation*}
Thus, for all $k\in\mathbb{N}$, there exists an element  $\bar{B}^{k+1} \in \partial\Psi(\mathbf{x}^{k+1})$ whose norm is upper bounded by:
\begin{equation*}
    \vertiii{\bar{B}^{k+1}} \leq \nu_k D_k.
\end{equation*}
\end{proof}
The proof of the convergence of the sequence requires the study of the limit points set, defined as follows. \medskip
\begin{definition}{\textbf{Limit points set \cite{bolte2014proximal}.}} \label{def:limit_point_set}
 The set of all limit points of sequences generated by $\acronym$ from a starting point $\mathbf{x}^0=\mathbf{\bar{x}}^0$ will be denoted by $\lp(\mathbf{\bar{x}}^0)$:
    \begin{align*}
        \lp(\mathbf{\bar{x}}^0) = \{\widehat{\mathbf{x}} \in \Hi, \exists & \text{ an increasing sequence of integers } \{k_j\}_{j\in\mathbb{N}},  \\ &\text{ such that } \mathbf{\bar x}^{k_j} \rightarrow \widehat{\mathbf{x}} \text{ as } j \rightarrow +\infty \}
    \end{align*}
\end{definition}
The properties of the limit points of sequences produced by some block algorithms are investigated in \cite{bolte2014proximal}, small modifications are required in our context. \medskip
\begin{lemma}{\textbf{Properties of the limit points set.}}
    \label{lm:limit_point_set}
    Suppose that Assumptions \ref{ass:deterministic1} and \ref{ass:deterministic2} hold. 
    Let $\{\mathbf{\bar{x}}^k\}_{k \in \mathbb{N}}$ be a sequence generated by Algorithm $\acronym$ starting from $\mathbf{\bar{x}}^0 = \mathbf{x}^0$. %
    The following hold:
    \begin{enumerate} 
        \item $\emptyset \neq \lp(\mathbf{\bar{x}}^0) \subset$ crit $\Psi$.
        \item We have 
        \begin{equation*}
            \lim_{k \rightarrow \infty} \mathrm{dist} (\mathbf{\bar{x}}^k,\lp(\mathbf{\bar{x}}^0)) = 0.
        \end{equation*}
        \item $\lp(\mathbf{\bar{x}}^0)$ is a nonempty, compact and connected set.
        \item The objective function $\Psi$ is finite and constant on $\lp(\mathbf{\bar{x}}^0)$.
    \end{enumerate}
\end{lemma}
\begin{proof}
    \noindent 1.%
    Let $\widehat{\mathbf{x}}$ be a limit point of $\{\mathbf{\bar{x}}^k\}_{k \in \mathbb{N}}$. By Definition \ref{def:limit_point_set} there exists a subsequence  $\{\mathbf{\bar{x}}^{k_q}\}_{q \in \mathbb{N}}$ such that $\mathbf{\bar{x}}^{k_q} \rightarrow \widehat{\mathbf{x}}$. Using assumption {\normalfont A$3$}, it follows that
    \begin{equation*}
        (\forall \ell \in \{1,\ldots,L\}), \quad \liminf_{q \rightarrow +\infty} g_\ell(x_\ell^{k_q}) \geq g_\ell(\hat{x}_\ell).
    \end{equation*}
    It is enough to prove that \begin{equation}\label{eq:upper_continuity}
        (\forall \ell \in \{1,\ldots,L\}), \quad  \limsup_{q \rightarrow +\infty}g_\ell(x_\ell^{k_q}) \leq g_\ell(\hat{x}_\ell).
    \end{equation}
    Indeed, since $f$ is continuous, \eqref{eq:upper_continuity} implies that
    \begin{equation*}
        \lim_{q\rightarrow \infty} \Psi(\mathbf{\bar{x}}^{k_q}) = \Psi(\widehat{\mathbf{x}}).
    \end{equation*}
    For all $k\in\mathbb{N}$ and for all $\ell\in\{1,\ldots,L\}$, we denote by $k_\ell\in\mathbb{N}$ the iteration in which block $\ell$ has received its last update in the $k$-th cycle. We have
    \begin{align*}
       (\forall \ell \in \{1,\ldots,L\}) \quad x_\ell^{k+1} = \argmin_{x_\ell \in \Hi_\ell}\left\{ \langle x_\ell-x_{\ell}^{k_\ell}, \nabla_\ell f(\mathbf{x}^{k_\ell}) \rangle + \frac{1}{2\tau_\ell^{k_\ell}} \Vert x_\ell- x_{\ell}^{k_\ell} \Vert^2 + g_\ell(x_\ell) \right\}.
    \end{align*}
    Thus, for $x_\ell = \widehat{x}_\ell$ it holds
    \begin{align*}
        \langle x_\ell^k-x_{\ell}^{k_\ell}, \nabla_\ell f(\mathbf{x}^{k_\ell}) \rangle + \frac{1}{2\tau_\ell^{k_\ell}} \Vert x_\ell^k- x_{\ell}^{k_\ell} \Vert^2 + g_\ell(x_\ell^k) \leq & \langle \widehat{x}_\ell-x_{\ell}^{k_\ell}, \nabla_\ell f(\mathbf{x}^{k_\ell}) \rangle \\ &+ \frac{1}{2\tau_\ell^{k_\ell}} \Vert \widehat{x}_\ell- x_{\ell}^{k_\ell} \Vert^2 + g_\ell(\widehat{x}_\ell).
    \end{align*}
    The index $k_\ell$ depends implicitly on $k$. Let $q\in\mathbb{N}$. For the rest of the proof we need to extract a converging subsequence, and to note the dependence to $q$ we will write $k_{q,\ell}$ to indicate the last update received by block $\ell$ at cycle $k_q$. %
    Taking then $k = k_q$, we obtain 
    \begin{align}
        & \langle x_\ell^{k_q}-x_\ell^{k_{q,\ell}}, \nabla_\ell f(\mathbf{x}^{k_{q,\ell}}) \rangle + \frac{1}{2\tau_\ell^{k_{q,\ell}}} \Vert x_\ell^{k_q}- x_\ell^{k_{q,\ell}} \Vert^2 + g_\ell(x_\ell^{k_q}) \nonumber \\& \leq \langle \widehat{x}_\ell-x_\ell^{k_{q,\ell}}, \nabla_\ell f(\mathbf{x}^{k_{q,\ell}}) \rangle + \frac{1}{2\tau_\ell^{k_{q,\ell}}} \Vert \widehat{x}_\ell- x_\ell^{k_{q,\ell}} \Vert^2 + g_\ell(\widehat{x}_\ell) \label{eq:upper_continuity_inequality}
    \end{align}
    Now we look at the limit when $q$ goes to infinity. Using the following  properties:
    \begin{itemize}
        \item[--] $\Vert x_\ell^{k_q}-x_\ell^{k_{q,\ell}}\Vert$ goes to $0$ as $q$ goes to infinity (Proposition \ref{prop:decrease}, item 1.),
        \item[--] $\nabla_\ell f$ is Lipschitz continuous and the sequence  $(\mathbf{x}^{n})_{n\in\mathbb{N}}$ is bounded (Assumption {\normalfont A$4$}), 
        \item[--] $\Vert \widehat{x}_\ell - x_\ell^{k_{q,\ell}} \Vert \leq \Vert \widehat{x}_\ell - x_\ell^{k_q} \Vert + \Vert x_\ell^{k_q} - x_\ell^{k_{q,\ell}} \Vert $ and both terms on the right-hand side of the inequality go to $0$ as $q$ goes to infinity.
    \end{itemize}
    Hence, \eqref{eq:upper_continuity_inequality} yields
    \begin{equation*}
        \limsup_{q\rightarrow +\infty} g_\ell(x_\ell^{k_q}) \leq g_\ell(\widehat{x}_\ell)
    \end{equation*}
    Now, combining points (1) and (2) of Proposition \ref{prop:decrease} that $\bar{B}^{k} \rightarrow 0$ as $k\rightarrow + \infty$.  The closedness property of $\partial \Psi$ \cite[Theorem 8.6]{VarAnalRockafellar} implies that $0\in\partial \Psi( \widehat{\mathbf{x}})$, and therefore that $\widehat{\mathbf{x}}$ is a critical point of $\Psi$. \medskip
    We have from Proposition \ref{prop:decrease}(1) and Lemma \ref{lm:norm_bound} that $\vertiii{ \mathbf{\bar x}^{k+1}-\mathbf{\bar{x}}^k} \rightarrow 0$ as $k\to+\infty$, % 
    and thus 2. and 3. hold (see \cite[Remark 5 \& Lemma 5]{bolte2014proximal}).
    4. is derived from 1. (\cite[Lemma 5]{bolte2014proximal}). 
    \end{proof}
Now that we have established the %
decrease of the objective function at each iteration and the properties of the limit point set, we are ready to state our main result. \medskip
\begin{theorem} \label{th:deterministic}
Suppose that Assumptions \ref{ass:deterministic1}, \ref{ass:deterministic2}, and \ref{ass:deterministic3} hold. Let $\{\mathbf{x}^n\}_{n\in\mathbb{N}}$, and $\{\mathbf{\bar{x}}^k\}_{k\in\mathbb{N}}$ be the sequences generated by Algorithm $\acronym$. Then,\medskip
\begin{enumerate} 
\item The sequence $\{\mathbf{\bar{x}}^k\}_{k\in\mathbb{N}}$ has finite length, that is, 
    \begin{equation*}
        \sum_{k=1}^\infty \vertiii{\mathbf{\bar{x}}^{k} -\mathbf{\bar{x}}^{k+1}} < \infty.
    \end{equation*}
    \item The sequence $\{\mathbf{\bar{x}}^k\}_{k\in\mathbb{N}}$ converges to a critical point $\widehat{\mathbf{x}}$ of $\Psi$.
    \end{enumerate}
\end{theorem}
\begin{proof}
\textbf{1.} For this crucial result, we follow the path of \cite{bolte2014proximal}. Since the sequence $(\mathbf{\bar{x}}^k)_{k\in\mathbb{N}}$ is bounded, there exists a sub-sequence that converges to $\widehat{\mathbf{x}}$.
        \begin{itemize}
            \item As $\{\Psi(\mathbf{\bar{x}}^k)\}_{k \in \mathbb{N}}$ is a non-increasing sequence, and as the limit points set $\lp(\mathbf{x}^0)$ is such that $\lim_{k \rightarrow \infty} \mathrm{dist} (\mathbf{\bar{x}}^k,\lp(\mathbf{x}^0)) = 0$ (Lemma \ref{lm:limit_point_set}, point (ii)), there exist $k_0 \in \mathbb{N},\varepsilon >0, \eta >0$ such that for all $k > k_0$, $\mathbf{\bar{x}}^k$ belongs to:
            \begin{equation*}
                \left\{\mathbf{x} \in \RR^N: \textrm{dist} (\mathbf{x},\lp(\mathbf{x}^0)) < \varepsilon \right\} \cap [\Psi(\widehat{\mathbf{x}}) < \Psi(\mathbf{x}) < \Psi(\widehat{\mathbf{x}}) + \eta].
            \end{equation*}
            \item Using now that $\lp(\mathbf{x}^0)$ is nonempty and compact, and that $\Psi$ is constant on it (Lemma \ref{lm:limit_point_set} (ii), (iii) and (iv)), by applying Lemma \ref{lm:unif_KL}, we obtain
            \begin{equation}
               (\forall k>k_0), \quad \varphi'\left(\Psi(\mathbf{\bar{x}}^k)-\Psi(\widehat{\mathbf{x}})\right) \textrm{dist}\left(0,\partial \Psi (\mathbf{\bar{x}}^k)\right) \geq 1, \text{ for some }\varphi \in \Phi_\eta \label{eq:KL_for_FLEX}
            \end{equation}
            \item Let $k\in\mathbb{N}$, Proposition \ref{prop:decrease} gives an upper bound of $\textrm{dist}(0,\partial \Psi (\mathbf{\bar{x}}^k))$. Set $\rho_1 = \min_{k\in\mathbb{N}} \rho_k$ and $\rho_2 = \max_{k\in\mathbb{N}} \nu_k$. We have
            \begin{equation*}
                \textrm{dist}(0,\partial \Psi (\mathbf{\bar{x}}^k)) \leq  \rho_2 D_{k-1}
            \end{equation*}
            Combined with \eqref{eq:KL_for_FLEX}, this bound  implies
            \begin{align} \label{eq:upper_bound} \varphi'(\Psi(\mathbf{\bar{x}}^k)-\Psi(\widehat{\mathbf{x}}))  \geq \rho_2^{-1}D_{k-1}^{-1}.
            \end{align}
            \item The concavity of $\varphi$ yields that:
            \begin{equation}
                \varphi(\Psi(\mathbf{\bar{x}}^k)-\Psi(\widehat{\mathbf{x}})) - \varphi(\Psi(\mathbf{\bar{x}}^{k+1})-\Psi(\widehat{\mathbf{x}})) \geq \varphi'(\Psi(\mathbf{\bar{x}}^k)-\Psi(\widehat{\mathbf{x}})) \left(\Psi(\mathbf{\bar{x}}^k)-\Psi(\mathbf{\bar{x}}^{k+1})\right). \label{eq:concavity}
            \end{equation}
            \item Now recall from the proof of Proposition \ref{prop:decrease}, item 1., \eqref{eq:Ck} that
            \begin{equation*}
                \rho_1 C_k
                \leq \Psi(\mathbf{\bar{x}}^k)-\Psi(\mathbf{\bar{x}}^{k+1}).
            \end{equation*}
            Define  
            \begin{equation*}
                (\forall p,q \in \mathbb{N}), \quad \Delta_{p,q} : = \varphi(\Psi(\mathbf{\bar{x}}^p)-\Psi(\widehat{\mathbf{\bar{x}}})) - \varphi(\Psi(\mathbf{\bar{x}}^{q})-\Psi(\widehat{\mathbf{\bar{x}}})),
            \end{equation*}
            By setting $\rho:= \rho_1 \rho_2^{-1} >0$, we have from \eqref{eq:upper_bound} and \eqref{eq:concavity}
            \begin{equation*}
                \Delta_{k,k+1} \geq \frac{\rho~C_{k}}{D_{k-1}}
            \end{equation*}
            and then:
            \begin{align*}
                C_k \leq \rho^{-1}~\Delta_{k,k+1} D_{k-1}
            \end{align*}
            \item We rewrite the expression of both $C_k$ and $D_k$
            \begin{equation*}
                D_k = \sum_{j=1}^{m_k} a_{k,j}, \text{ and }  C_k = \sum_{j=1}^{m_k} a_{k,j}^2.
            \end{equation*}
            where $m_k = \sum_{n=k\cdot K}^{(k+1)} \textrm{card}(\mathbb{I}^n)$ and $a_{k,j} = \Vert x_\ell^n- x_\ell^{n+1} \Vert$ where $j$ browses through $n$ and $\ell$ by increasing order. Then,
            \begin{align*}
                D_k \leq \sqrt{m_k C_k}
            \end{align*}
            Using $2\sqrt{ab}\leq a + b$ for all $a,b\geq 0$, and writing $M = \max_k \sqrt{m_k}$ we get that
            \begin{equation*}
                2 D_k \leq M \rho^{-1} \Delta_{k,k+1} + D_{k-1}
            \end{equation*}
            We have then
            \begin{align*}
                    2 \sum_{i=k_0+1}^{k} D_i & \leq  \sum_{i=k_0+1}^{k} D_{i-1} + M \rho^{-1} \sum_{i=k_0+1}^k \Delta_{i,i+1} \nonumber \\
                &  = \sum_{i=k_0}^{k-1} D_{i} + M \rho^{-1} \sum_{i=k_0+1}^k \Delta_{i,i+1} \nonumber \\
                &  \leq \sum_{i=k_0+1}^{k} D_i + D_{k_0} + M \rho^{-1} \sum_{i=k_0+1}^k \Delta_{i,i+1} \nonumber \\
                \implies  \sum_{i=k_0+1}^{k} D_i & \leq D_{k_0} + M \rho^{-1} \sum_{i=k_0+1}^k \Delta_{i,i+1} \nonumber \\
                \implies  \sum_{i=k_0+1}^{k} D_i & \leq D_{k_0} + M \rho^{-1} \Delta_{k_0+1,k+1},
            \end{align*}
            the last line coming from the fact $\Delta_{p,q}+\Delta_{q,r} = \Delta_{p,r}$ for all $p,q,r\in\mathbb{N}$ \cite[Proof of theorem 3.1]{bolte2014proximal}. As $\varphi \geq 0$, we have that:
            \begin{equation*}
                \Delta_{k_0+1,k+1} = \varphi(\Psi(\mathbf{\bar{x}}^{k_0+1})-\Psi(\widehat{\mathbf{x}})) - \varphi(\Psi(\mathbf{\bar{x}}^{k+1})-\Psi(\widehat{\mathbf{x}})) \leq \varphi(\Psi(\mathbf{\bar{x}}^{k_0+1})-\Psi(\widehat{\mathbf{x}}))
            \end{equation*}
            Then as $D_k \geq \vertiii{\mathbf{\bar{x}}^{k}-\mathbf{\bar{x}}^{k+1}}$ (Lemma \ref{lm:norm_bound}), this allows us then to conclude then that $\{\mathbf{\bar{x}}^k\}_{k \in \mathbb{N}}$ has finite length:
            \begin{equation*}
                \sum_{k=1}^{\infty} \vertiii{\mathbf{\bar{x}}^{k}-\mathbf{\bar{x}}^{k+1}}<\infty.
            \end{equation*}
        \end{itemize}
\textbf{2.} The finite length of the sequence implies that it is a Cauchy sequence and hence a convergent sequence \cite[Proof of Theorem 3.1(ii)]{bolte2014proximal}.
\end{proof}
    We  splitted the proof of our main result into $4$ steps, which are common when studying descent algorithm on KŁ functions \cite{bolte2014proximal,attouch2009convergence,attouch2010proximal,xu2017globally,repetti2021variable}. For each step we detail how the existing proofs were adapted for our approach by highlighting the difference with the existing literature.%
\begin{enumerate} 
    \item \textbf{Sufficient decrease property:}  at each cycle, the objective function $\Psi$ is decreased. The decrease is controlled by the squared norm of the differences between the block updates. \textit{Difference with the literature:} we introduce the possibility of parallel updates to decrease the function, and thus to adapt the choice of step size to the smoothness of the group of blocks considered.
    \item \textbf{Subgradient upper bound:} at each cycle, we can exhibit an upper bound on one element of the subgradient of $\Psi$ at the cycle iterate. This upper bound is controlled by the norm of the differences between the block updates. \textit{Difference with the literature:} this bound is less sharp than in the literature (as the reader can see in the proof), but it is necessary to write it in this way to apply the KŁ property and obtain finite length. %
    \item \textbf{Limit points are critical points:} the set of limit points of the sequences generated by our algorithm will be a subset of the set of critical points of $\Psi$. \textit{Difference with the literature:} due to the possible parallel block updates, we need to be a bit more cautious when looking at the converging subsequences.
    \item \textbf{Finite length of the sequences:} the sequences generated by our algorithm have finite length and thus converge. This is a consequence of the KŁ property satisfied by $\Psi$ (see Lemma \ref{def:kl}) and of points (i) and (ii). \textit{Difference with the literature:} we invoke a particular instance of Cauchy-Schwartz inequality to obtain the desired result.
\end{enumerate}
\begin{remark} Regarding Theorem \ref{th:deterministic}, it may not appear obvious that algorithm $\acronym$ needs to update each block in an essentially cyclic manner, but this is indeed necessary, otherwise the norm of $\bar{B}^k\in\partial\Psi(\mathbf{\bar{x}}^k)$ would not go to $0$ as $k$ goes to infinity.

\end{remark}
\subsection{Some consequences.}
Our main result implies that the sequence of iterates $\{\mathbf{x}^n\}_{n\in\mathbb{N}}$ also converges, and to establish some convergence rates depending on the KŁ property.
\paragraph{Convergence of the sequence $\{\mathbf{x}^n\}_{n\in\mathbb{N}}$.} We have established the convergence of the sequence of cycle iterates $\{\mathbf{\bar x}^k\}_{n\in\mathbb{N}}$ to a critical point while decreasing the objective function. A direct consequence is that the sequence $\{\mathbf{x}^n\}_{n\in\mathbb{N}}$ converges to the same critical point.
\begin{proposition}
Suppose that Assumptions \ref{ass:deterministic1}, \ref{ass:deterministic2}, and \ref{ass:deterministic3} hold. Let $\{\mathbf{x}^n\}_{n\in\mathbb{N}}, \{\mathbf{\bar{x}}^k\}_{k\in\mathbb{N}}$ be generated by Algorithm $\acronym$. Let $\mathbf{\widehat{x}}\in$ crit $\Psi$ such that $\mathbf{\bar{x}}^k \to_{k \to +\infty} \mathbf{\widehat x}$. Then,
\begin{equation*}
    \lim_{n\rightarrow+\infty} \mathbf{x}^n = \mathbf{\widehat{x}}.
\end{equation*}
\end{proposition}
\begin{proof}
    With a similar argument than in Lemma \ref{lm:norm_bound}, we have
    \begin{equation*}
        (\forall k \in \mathbb{N}), \forall n_0 \in \mathbb{N} \cap [k\cdot K, (k+1) \cdot K-1]), \quad \vertiii{\mathbf{x}^{n_0}-\mathbf{\bar{x}}^k} \leq \sum_{n= k \cdot K}^{n_0-1} \sum_{\ell \in \mathbb{I}^n} \Vert x_\ell^{n+1} - x_\ell^n \Vert. 
    \end{equation*}
    Combining the sufficient decrease property (Proposition \ref{prop:decrease}(1)) of our algorithm and the previous bound, we have
    \begin{equation*}
        (\forall n_0 \in \mathbb{N} \cap [k\cdot K, (k+1) \cdot K-1]), \quad \lim_{k \rightarrow +\infty} \vertiii{\mathbf{x}^{n_0}-\mathbf{\bar{x}}^k} = 0.
    \end{equation*}
    We showed that $\{\mathbf{\bar{x}}^k\}_{k\in\mathbb{N}}$ converges to $\mathbf{\widehat{x}}$, hence by definition
    \begin{equation*}
        \forall \xi_1 >0, \exists~k_1 \in \mathbb{N}, \text{ such that for all } k\geq k_1, \vertiii{\mathbf{\bar{x}}^k-\mathbf{\widehat{x}}} <\xi_1.
    \end{equation*}
    Now choose $\xi_2>0$ such that for all $k\geq k_2 := k_1$, $\vertiii{\mathbf{x}^{n_0}-\mathbf{\bar{x}}^k} < \xi_2$. We have
    \begin{align*}
        \vertiii{\mathbf{x}^{n_0}-\mathbf{\widehat{x}}} & = \vertiii{\mathbf{x}^{n_0}-\mathbf{\bar{x}}^k + \mathbf{\bar{x}}^k-\mathbf{\widehat{x}}} \nonumber \\ 
        & \leq \vertiii{\mathbf{x}^{n_0}-\mathbf{\bar{x}}^k} + \vertiii{\mathbf{\bar{x}}^k-\mathbf{\widehat{x}}} \nonumber\\
        & \leq \xi_1+\xi_2
    \end{align*}
    To prove that $\forall \xi>0$, there exists $N\in\mathbb{N}$, such that for all $n\geq N$, $\vertiii{\mathbf{x}^{n_0}-\mathbf{\widehat{x}}} < \xi$, we can take $\xi_1 = \xi_2 := \xi/2$ and take $N = \max\{k_1,k_2\} \cdot K$ to conclude.
\end{proof}
\paragraph{Convergence rates.} As developed in \cite{attouch2009convergence,bolte2014proximal}, the case in which all the functions involved are semi-algebraic is an interesting one for the study of the convergence. The desingularizing function $\varphi$ can be chosen to be of the form \cite{attouch2009convergence}
$
    \varphi(s) = cs^{1-\theta},
$
where $c>0$ and $\theta \in [0,1)$. Then, depending on the value of $\theta$, the following convergence rates hold: \cite{attouch2009convergence,bolte2014proximal}
\begin{enumerate}
    \item If $\theta = 0$  the sequence $(\bar{\mathbf{x}}^k)_{k\in\mathbb{N}}$ converges in a finite number of steps. \vspace{1em}
    \item If $\theta \in (0,1/2]$  there exist $\omega>0$ and $\kappa \in [0,1)$ such that $\vertiii{\bar{\mathbf{x}}^k-\widehat{\mathbf{x}}}\leq \omega \kappa^k$. \vspace{0.8em}
    \item If $\theta \in (1/2,1)$  there exists $\omega>0$ such that $\vertiii{\bar{\mathbf{x}}^k-\widehat{\mathbf{x}}}\leq \omega k^{-\frac{1-\theta}{2\theta -1}}$.
\end{enumerate}

\subsection{Convexity of the regularization}

We can derive a slightly different sufficient decrease property of our algorithm when assuming convexity of the regularizing functions $g_\ell$, for all $\ell$. This assumption allows us to take bigger step sizes when updating the blocks. \medskip

\begin{lemma}{\textbf{Sufficient decrease property: convexity of $g_\ell$.}} \label{lm:sufficient_decrease_convex} Suppose that Assumption \ref{ass:deterministic2} holds, and that for all $\ell \in \{1,\ldots,L\}$, $g_\ell$ is convex. Let $\{\mathbf{\bar{x}}^k\}_{k\in\mathbb{N}}$, and $\{\mathbf{x}^n\}_{n\in\mathbb{N}}$ be generated by algorithm $\acronym$.%
For each $n$, let $\beta_f^n = \sqrt{\sum_{j\in J_n,1\leq \ell \leq L} \varepsilon_j \beta_{\ell,j}^2}$ and $0<\tau_\ell^n < 2/\beta_f^n$.  Then 
\begin{align*}  \Psi(\mathbf{\bar{x}}^{k+1})+
        \left(\sum_{n=k\cdot K}^{(k+1) \cdot K-1} \sum_{\ell \in \mathbb{I}^n} \left(\frac{1}{\tau_\ell^n}-\frac{\beta_f^n}{2} \right)\Vert x_\ell^n-x_\ell^{n+1} \Vert^2 \right) &\leq
        \Psi(\mathbf{\bar{x}}^k).
        \end{align*}
\end{lemma}
\begin{proof}
Due to the convexity of $g_\ell$ for all $\ell$, the associated proximity operator is single-valued.
For all $\ell \in \mathbb{I}^n$, the first order optimality conditions of the proximity operator \eqref{eq:prox_equation} yield:
\begin{equation}\label{eq5:first_order_suff_decrease_convex}
    g_\ell(x_\ell^{n+1}) + \frac{1}{\tau_\ell^n} \Vert x_\ell^n-x_\ell^{n+1} \Vert^2 \leq g(x_\ell^n) + \langle \nabla_\ell f(\mathbf{x}^{n}), x_\ell^n-x_\ell^{n+1} \rangle.
\end{equation}
The subtle difference with the non-convex case is the factor dividing $\Vert x_\ell^n-x_\ell^{n+1} \Vert^2$. With the same derivation as in the non-convex case (proof of point (i) in Theorem \ref{th:deterministic}), we obtain finally that:
\begin{align*}  \Psi(\mathbf{\bar{x}}^{k+1})+
    \left(\sum_{n=k\cdot K}^{(k+1) \cdot K-1} \sum_{\ell \in \mathbb{I}^n} \left(\frac{1}{\tau_\ell^n}-\frac{\beta_f^n}{2} \right)\Vert x_\ell^n-x_\ell^{n+1} \Vert^2 \right) &\leq
    \Psi(\mathbf{\bar{x}}^k).
\end{align*}
\end{proof}

\section{Multilevel forward-backward as a flexible block-coordinate forward-backward}
\label{sec5:motivating_example}

We have seen how to construct a convergent parallel and essentially cyclic block-coordinate forward-backward algorithm, able to handle non-convexity and non-smoothness of the objective function. This algorithm is deterministic by essence, even though a random shuffle of the order of the updates is possible.

Our algorithm is compatible with several update rules. In particular, $\acronym$ can use a hierarchical update rule, to update blocks given their position in a hierarchy inherited from the function to optimize. Such hierarchy arises for instance in image restoration problems when the regularization penalizes wavelet coefficients of the image. Such hierarchical update rule has a direct connection with multilevel algorithms.

We start by recalling some key facts about multiresolution analysis to introduce the rigorous use of wavelets made in this chapter; then present the construction of the block algorithm for two blocks; and then discuss the construction of the multilevel algorithm for two levels. This presentation will allow us to highlight that the two approaches can be rigorously equivalent in this setting. This equivalence opens practical applications for BCD methods and theoretical insights for multilevel methods.  %
\subsection{Multiresolution analysis in a nutshell} 
Let $\mathbf{u} \in L_2(\Omega)$ be an image, where $\Omega \subset \RR^2$. We decompose $L_2(\Omega)$ into the sum of the space of approximation coefficients $V_J$, and that of detail coefficients $V_J^\perp$ at resolution $J \in \mathbb{N}$ \cite[Chapter 7]{mallat1999wavelet}, i.e.,
\begin{equation*}
    L_2(\Omega) = V_J \oplus V_J^\perp.
\end{equation*}
We assume that $\mathbf{u}$ lives exclusively in $V_J$ in the following, i.e., $\mathbf{u}$ is a discrete image of ${2^{2^J}}$ pixels. $V_J$ can be decomposed into subspaces $V_{J-1}$ and $W_{J-1}$, where $V_{J-1}$ is the space of approximation coefficients at resolution $J-1$ and $W_{J-1}$ the space of detail coefficients at resolution $J-1$. More precisely, 
\begin{equation*}
    \mathbf{u} = \Pi^*_{V_{J-1}} a_{J-1} + \Pi^*_{W_{J-1}} d_{J-1},
\end{equation*}
where 
$a_{J-1}$ (resp. $d_{J-1}$) are the approximation (resp. detail) coefficients at resolution $J-1$ and $\Pi_{V_{J-1}}$ (resp. $\Pi_{W_{J-1}}$) is the linear projection operator onto $V_{J-1}$ (resp. $W_{J-1}$). By definition, $\Pi_{V_{J-1}} \Pi^*_{V_{J-1}} = \Id_{V_{J-1}}$ and $\Pi_{W_{J-1}} \Pi^*_{W_{J-1}} = \Id_{W_{J-1}}$, where $\Id_{V_{J-1}}$ and $\Id_{W_{J-1}}$ are the identity operators on $V_{J-1}$ and $W_{J-1}$, respectively. Note that the block $d$ contains the three groups of detail coefficients \cite{mallat1999wavelet}.
\subsection{Wavelet deblurring with multilevel FB: a hierarchical block algorithm}
Our $2^{2^J}$ pixels image $\mathbf{u}$ is decomposed into two independent components $a_{J-1} \in V_{J-1} $, and $d_{J-1} \in W_{J-1}$. In this setting the blocks $(x_1,x_2)$ are equal to $(a_{J-1},d_{J-1})$. We drop the index $J-1$ in the following for simplicity, so that $(x_1,x_2)=(a,d) \in V \times W$. 
\begin{figure}
    \centering
    \includegraphics[width=0.3\textwidth]{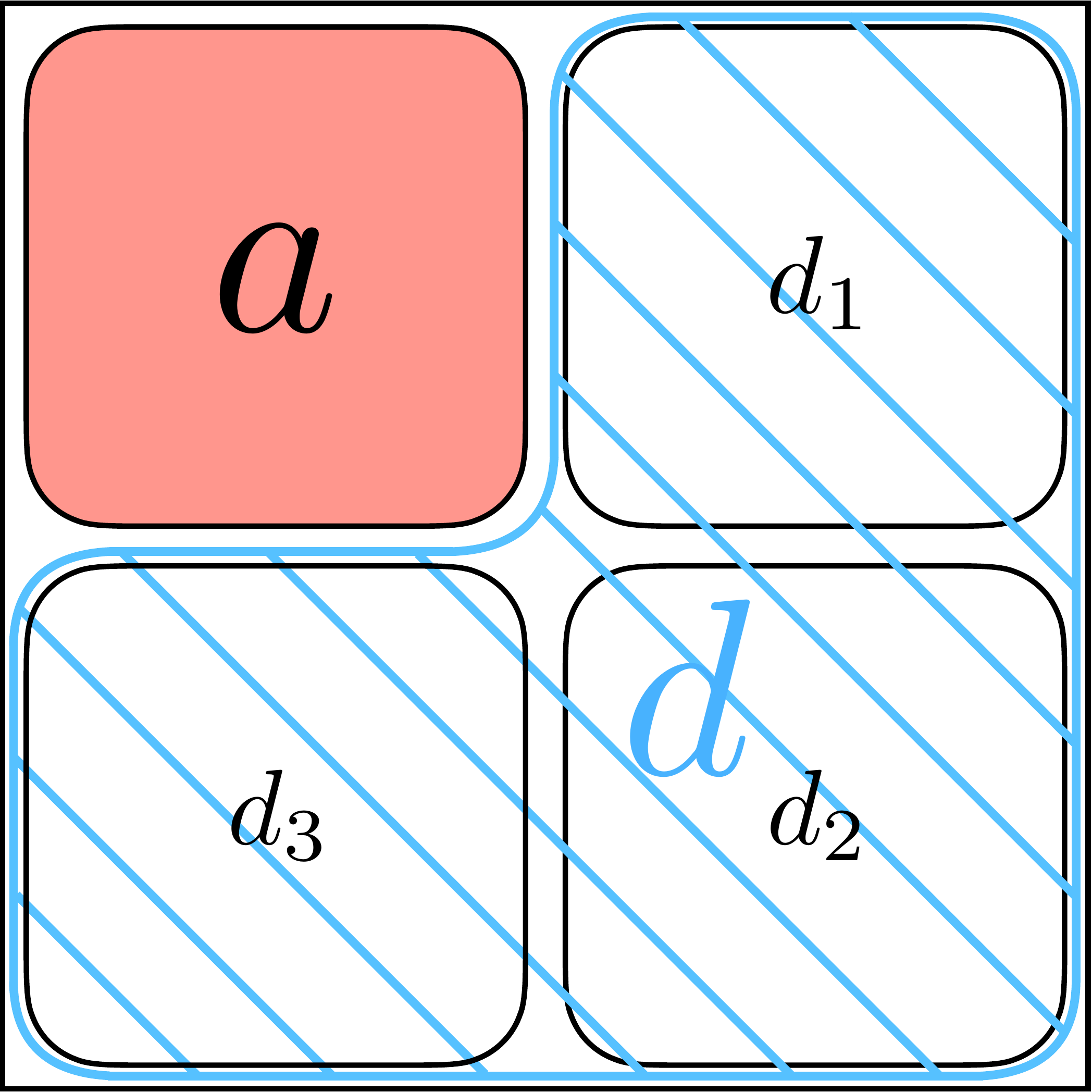}
    \caption{Decomposition on two levels of an image $\mathbf{u}$ with a wavelet transform. We regroup the detail coefficients $d_1,d_2$ and $d_3$ into one single block {\color{lightblue}$d$} to simplify the presentation of our two level or two block proximal gradient descent algorithms. }
    \label{fig5:two_level_wavelet}
\end{figure}
We aim at solving the following optimization problem:
\begin{equation}
    \Argmin_{\mathbf{u} \in V_J} F(\mathbf{u}) = \frac{1}{2} \Vert \A \mathbf{u} -\mathbf{z} \Vert_2^2 +\lambda g(\D \mathbf{u}), \label{eq5:wavelet_optim_2L}
\end{equation} 
where $\A:V_J \to V_J$ is a bounded linear operator modelling a blur, $\mathbf{z}$ is a degraded version of $\mathbf{u}$, $\D$ is the wavelet transform of $\mathbf{u}$ on two levels, and $\lambda$ is a multi-valued parameter to penalize differently the approximation and detail coefficients. We also assume that $g$ is separable along $a$ ($g_a$) and $d$ ($g_d$), and proximable, hence $g\circ\D$ is also proximable. %
We can rewrite this classical wavelet penalized least-squares problem using the wavelet decomposition of $\mathbf{u}$, i.e., the synthesis formulation
\begin{align} \label{eq5:wavelet_optim_2B}
    \left(\widehat{a},\widehat{d}\right) \in \Argmin_{a\in V, d\in W}\Psi(a,d) & =\frac{1}{2} \Vert \A\left(\Pi_V^* a + \Pi_W^* d\right) -\mathbf{z}\Vert^2 + \lambda_a g_a(a) + \lambda_d g_d(d).
\end{align}
To find a minimizer of $\Psi$ with respect to ($a$,$d$) is  equivalent to finding a minimizer of $F$ with respect to $\mathbf{u}$, as we can recover the solution $\mathbf{\widehat{u}}$ of Problem \eqref{eq5:wavelet_optim_2L} from the solution $(\widehat{a},\widehat{d})$ of Problem \eqref{eq5:wavelet_optim_2B}, by setting
\begin{equation*} 
    \widehat{\mathbf{u}} = \Pi_V^* \widehat{a} + \Pi_W^* \widehat{d}.
\end{equation*}
In the following, note that $D = \left[ \begin{array}{c}
    \Pi_V  \\
    \Pi_W  
\end{array}\right]$.
\paragraph{Two block-coordinate proximal gradient descent.}
To minimize $\Psi$, we first consider the block-coordinate approach. Given $(a^0,d^0)\in V \times W$, and $0<\tau< 2/ \Vert \A^* \A \Vert$, the iterations read 
\begin{equation}\label{eq5:example_BC_FB}
    \begin{array}{l}
    \text{for } n=0,1,\dots \\ 
    \left\lfloor
    \begin{array}{l}
        a^{n+1} = a^n + \varepsilon_{a,n} \left(\prox_{\tau \lambda_a g_a} \left(a^n - \tau \Pi_V \A^* \left[\A\left(\Pi_V^* a^n + \Pi_W^* d^n \right) -\mathbf{z} \right]\right) -a^n \right)\\
        d^{n+1} = d^n + \varepsilon_{d,n} \left(\prox_{\tau \lambda_d g_d} \left( d^n - \tau \Pi_W \A^* \left[ \A \left(\Pi_V^* a^n + \Pi_W^* d^n\right) -\mathbf{z} \right] \right) -d^n \right)
    \end{array}
    \right.
    \end{array}
\end{equation}
where for every $n\in\mathbb{N}$, $(\varepsilon_{a,n},\varepsilon_{d,n}) \in \{0,1\}^2$. Note that if $(\varepsilon_{a,n},\varepsilon_{d,n})=(1,1)$ for all $n$, \eqref{eq5:example_BC_FB} reduces to the standard forward-backward algorithm. A cyclic coordinate descent algorithm is obtained by setting alternatively one of $\varepsilon_{a,n},\varepsilon_{d,n}$ to $1$. This could also be set at random, provided that $\mathsf{P}\left[(\varepsilon_{a,n},\varepsilon_{d,n})=(0,0)\right] = 0$.

\paragraph{Two-level proximal gradient descent.}
We present now the construction of a two-level proximal algorithm to minimize $\Psi$ (the fine level function in the multilevel terminology \cite{lauga_iml_2023}). We will denote by $\Psi_H$ an approximation of $\Psi$, deemed the coarse level function. Given the structure of the problem, it is natural to define $\Psi_H$ in the approximation space $V$. Consequently, and following \cite{lauga_iml_2023}, the information transfer operator $\IhH $ (that sends information from the fine level to the coarse level) is the restriction $R_V$ onto $V$ (i.e., $R_V(a,d) = a$) and the prolongation operator $\IHh$ (that sends information from the coarse level to the fine level) is directly $R_V^*$ (i.e., $R_V^* a = (a,0)$).

By setting $\A_H = \A \Pi_V^*$, the coarse model is chosen as:
\begin{equation}\label{eq5:example_coarse_model}
    \Psi_H(a) = \frac{1}{2} \Vert \A_H a - \Pi_V^* \Pi_V \mathbf{z} \Vert_2^2 + \lambda_a g_a(a) + \langle v_H, a \rangle,
\end{equation}
where $v_H$ enforces the first order coherence \cite[Definition 2.1]{lauga_iml_2023} between two smoothed versions $\Psi_\mu$ and $\Psi_{H,\mu}$ of $\Psi$ and $\Psi_H$ respectively \cite[Definition 2.4]{lauga_iml_2023}, with parameter $\mu>0$ (i.e., the smoothed coarse level function to be the first order Taylor approximation of the smoothed fine level function):
\begin{equation*}
    v_H = R_V \nabla \Psi_\mu(a^n,d^n) - \nabla \Psi_{H,\mu}(a^n),
\end{equation*}
with $\Psi_{H,\mu}(\cdot)=\Vert \A_H \cdot - \Pi_V^* \Pi_V \mathbf{z}\Vert_2^2 +g_{a,\mu}(\cdot),$ and $\Psi_\mu(\cdot) = \Vert \A \cdot - \mathbf{z} \Vert_2^2 + g_{a,\mu}(\cdot) +g_{d,\mu}(\cdot)$, where we have denoted $g_{\cdot,\mu}$ the $\mu>0$-smoothed $\ell_1$-norm (according to the principles of \cite{beck2012}).
To go from one level to the other, multilevel algorithm employ information transfer operators. 

In the following, we assume that we compute only one coarse iteration before going back to the fine level, but everything holds trivially for more coarse iterations. This iteration will yield $a^{n+1/2}$ from $a^n$. The coarse level model being non-smooth, we will use a proximal gradient step to optimize it. Accordingly, the two-level proximal gradient algorithm, starting from $(a_0,d_0)\in V \times W$, is
\begin{equation}\label{eq5:example_multilevel}
    \begin{array}{l}
    \text{for } n=0,1,\dots \\ 
    \left\lfloor
    \begin{array}{l}
        a^{n+1/2} = \prox_{\tau \lambda_a g_a} \left(a^n - \tau \A_H^* \left( \A_H a^n - \Pi_V^* \Pi_V \mathbf{z} \right) - \tau v_H \right) \\
        a^{n+1} = \prox_{\tau \lambda_a g_a} \left(a^{n+1/2} - \tau \Pi_V \A^* \left(\A\left(\Pi_V^* a^{n+1/2} + \Pi_W^* d^n \right) -\mathbf{z} \right)\right)\\
        d^{n+1} = \prox_{\tau \lambda_d g_d} \left( d^n - \tau \Pi_W \A^* \left( \A \left(\Pi_V^* a^{n+1/2} + \Pi_W^* d^n\right) -\mathbf{z} \right) \right) \\
    \end{array}
    \right.
    \end{array}
\end{equation}
The fact that Algorithms \eqref{eq5:example_BC_FB} and \eqref{eq5:example_multilevel} are the same algorithm is not obvious at first sight. We will show in the following that this is indeed the case. After summarizing our assumptions, we will compute the first order coherence term explicitly.\medskip
\begin{assumption} \label{ass5:simple_example}
    We assume that:
    \begin{enumerate}[label=(\roman*)]
        \item the information transfer operator is the restriction $R_V$ onto $V$;
        \item in the definition of $v_H$, the fine and coarse models are smoothed with the same smoothing technique, with the same smoothing parameter $\mu>0$;
        \item $\Psi$ and $\Psi_H$ are first order coherent with respect to their smoothed versions  \cite[Definition 2.1]{lauga_iml_2023}.
    \end{enumerate}
\end{assumption} \medskip

\begin{lemma} \label{lm:FOC}
Suppose that Assumption \ref{ass5:simple_example} holds. The first order coherence term $v_H$ in \eqref{eq5:example_coarse_model} at $(a^n,d^n)$ is given by:
\begin{equation}\label{eq5:proof_lemma_FOC_example_1}
    v_H = \Pi_V \A^* \left(\A  \Pi_W^* d^n - \Pi_W^* \Pi_W \mathbf{z} \right).
\end{equation}
The first order coherence sends the contribution of the detail coefficients to the gradient of the data fidelity term to the coarse level.
\end{lemma}
\begin{proof}
    By definition of first order coherence between smoothed functions \cite[Definition 2.1]{lauga_iml_2023}, we have:
    \begin{equation}\label{eq5:proof_lemma_FOC_example_2}
        v_H = R_V \nabla \Psi_\mu(a^n,d^n) - \nabla \Psi_{H,\mu}(a^n).
    \end{equation}
    The second term is the simple computation of the gradient of the coarse model:
    \begin{align}
        \nabla \Psi_{H,\mu}(a^n) & = \nabla \left( \frac{1}{2} \Vert \A_H a^n -\Pi_V^*\Pi_V \mathbf{z} \Vert_2^2 + \lambda_a g_{a,\mu}(a^n) \right) \nonumber \\
        & = \A_H^* \left( \A_H a^n - \Pi_V^*\Pi_V \mathbf{z} \right) + \lambda_a \nabla_{a} g_{a,\mu} (a^n) \nonumber \\
        & = \Pi_V \A^* \left( \A\Pi_V^* a^n - \Pi_V^*\Pi_V \mathbf{z} \right) + \lambda_a \nabla_{a} g_{a,\mu}(a^n).
    \end{align}
    On the other hand,
    \begin{align}
        \nabla \Psi_\mu(a^n,d^n) & = \nabla \left( \frac{1}{2} \Vert \A \left( \Pi_V^* a^n + \Pi_W^* d^n \right) - z \Vert_2^2 + \lambda_a g_{a,\mu}(a^n) + \lambda_d g_{d,\mu}(d^n) \right) \nonumber\\
        & =  \begin{bmatrix} \Pi_V \A^* \left( \A \left( \Pi_V^* a^n + \Pi_W^* d^n \right) - z \right) + \lambda_a \nabla_{a} g_{a,\mu} (a^n) \\ \Pi_W \A^* \left( \A \left( \Pi_V^* a^n + \Pi_W^* d^n \right) - z \right) + \lambda_d \nabla_{d} g_{d,\mu} (d^n) \end{bmatrix}, \nonumber
        %& = \A^* \left( \A \left( \Pi_V^* a^n + \Pi_W^* d^n \right) - z \right) + \lambda_a \nabla_{a} \left(\Vert \cdot \Vert_{1,\mu} \right)(a^n) + \lambda_d \nabla_{d} \left(\Vert \cdot \Vert_{1,\mu} \right)(d^n) \nonumber
    \end{align}
    and thus:
    \begin{align*}
        v_H & = R_V \left(\begin{bmatrix} \Pi_V \A^* \left( \A \left( \Pi_V^* a^n + \Pi_W^* d^n \right) - \mathbf{z} \right) + \lambda_a \nabla_{a} g_{a,\mu} (a^n) \\ \Pi_W \A^* \left( \A \left( \Pi_V^* a^n + \Pi_W^* d^n \right) - \mathbf{z} \right) + \lambda_d \nabla_{d} g_{d,\mu} (d^n) \end{bmatrix} \right)\\
        & - \Pi_V \A^* \left( \A\Pi_V^* a^n - \Pi_V^*\Pi_V \mathbf{z} \right) + \lambda_a \nabla_{a} g_{a,\mu} (a^n).
    \end{align*}

    Since $V$ and $W$ are orthogonal to each other, restricting any element of $W$ on $V$ yields $0$ and thus:
    \begin{equation*}
        v_H = \Pi_V \A^* \left(\A  \Pi_W^* d^n - \Pi_W^* \Pi_W \mathbf{z} \right),
    \end{equation*}
    where we used that 
    \begin{equation*}
        \mathbf{z} - \Pi_V^* \Pi_V \mathbf{z} = \Pi_W^* \Pi_W \mathbf{z}.
    \end{equation*}
\end{proof}
Based on Lemma \ref{lm:FOC}, the proximal gradient step at coarse level at iteration $n$ reads
\begin{align}
    a^{n+1/2} & = \prox_{\tau \lambda_a g_a} \left(a^n - \tau \A_H^* \left( \A_H a^n - \Pi_V^* \Pi_V \mathbf{z} \right) - \tau v_H \right) \nonumber \\
    & = \prox_{\tau \lambda_a g_a} \left(a^n - \tau \Pi_V \A^* \left( \A \Pi_V^* a^n  -\mathbf{z}  + \A \Pi_W^* d^n \right) \right) \nonumber \\
    & = \prox_{\tau \lambda_a g_a} \left(a^n - \tau \Pi_V \A^* \left( \A \left(\Pi_V^* a^n + \Pi_W^* d^n \right) -\mathbf{z} \right) \right),
\end{align}
which fits a block-coordinate update on $a$.
\begin{figure}
    \centering
    \includegraphics[trim={21em 45em 90em 20em},clip,width=0.6\textwidth]{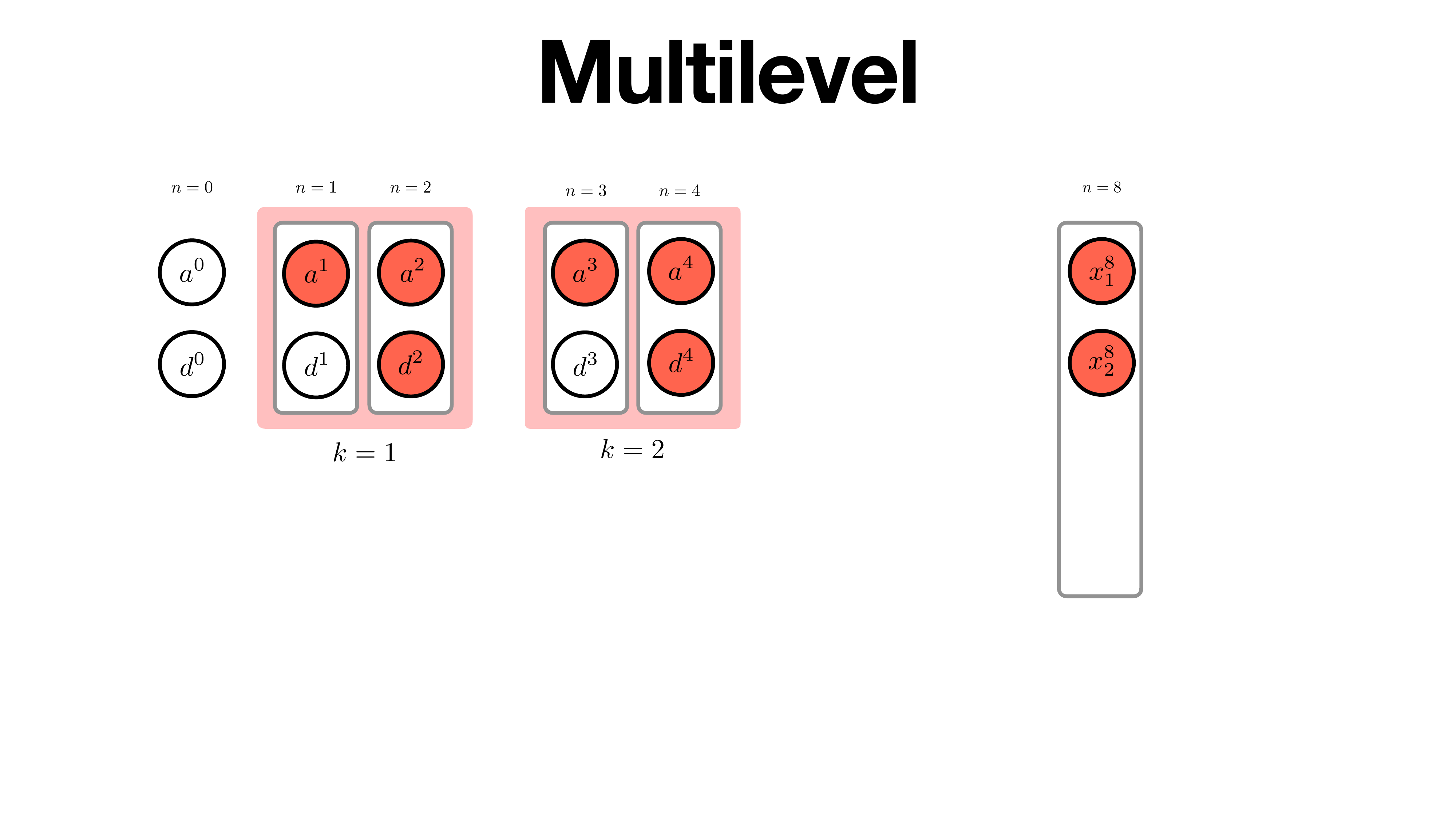}
    \caption{Update scheme of the two block-coordinate descent algorithm. The blocks are updated in a cyclic fashion, first with the approximation block updated alone ($a^1$ in \textcolor{red}{red}), then the approximation and detail blocks updated together ($a^2$ and $d^2$ in \textcolor{red}{red}). We represent two cycles in this figure $k=1$ and $k=2$ for a total of $n=4$ iterations. 
    }
    \label{fig5:two_block_wavelet}
\end{figure}
We summarize the consequence of this result in the following proposition: \medskip
\begin{proposition} \label{lemma5:equivalence_example}
The two-level algorithm defined in \eqref{eq5:example_multilevel} is equivalent to the block-coordinate algorithm defined in \eqref{eq5:example_BC_FB} when choosing $\varepsilon_{a,2n} = 1$ and $\varepsilon_{d,2n} = 0$, then $\varepsilon_{a,2n+1} = 1$ and $\varepsilon_{d,2n+1} = 1$ for all $n \in \mathbb{N}$ and the same initial vector $(a^0,d^0) \in V \times W$.
\end{proposition}

A generalization of this argument to an arbitrary number of levels can be found in \cite{lauga_multilevel_2024}.

Multilevel algorithms have been shown to accelerate the resolution of image restoration problems in the literature \cite{javaherian2017,fung2020,plier2021,lauga_iml_2023,lauga2022fista,lauga2024radio,parpas2017}. Therefore following update rules from the multilevel literature is of gret interest for block algorithms (see example (4) in Figure \ref{fig:updates_scheme} and Figure \ref{fig5:two_block_wavelet}).

\section{Numerical experiments} \label{sec:numerical}
In this section, we present numerical experiments to assess the performance of the proposed construction of a block-coordinate descent algorithm mimicking multilevel iterations. We show that by emulating the behavior of multilevel algorithms, our $\acronym$ algorithm has superior practical performance with respect to the other possible choices of update rules. We also show that our algorithm outperforms the standard FB algorithm. %

\paragraph{Optimization problem.}
Consider the optimization problem
\begin{equation}
    \Argmin_{\mathbf{u} \in \RR^N} F(\mathbf{u}) = \frac{1}{2} \Vert \A \mathbf{u} -\mathbf{z} \Vert_2^2 +\lambda g(\D \mathbf{u}), \label{eq5:wavelet_optim_exp}
\end{equation} 
where $\A$ encodes a Gaussian blur, and $\D$ is the 2 levels wavelet decomposition, i.e., $g(\D \cdot)$ will penalize the approximation coefficients, and the three blocks of detail coefficients associated with $\mathbf{u}$ (see Figure \ref{fig5:two_level_wavelet}). We set
\begin{equation*}
    (\forall (a,d) \in V \times W), \quad g(a,d) = \lambda_a\sum_{i} \log(| a_i | + \epsilon) + \lambda_d\sum_{i} \log(| d_i | + \epsilon),
\end{equation*}
where $a_i,d_i$ denote the components of the vectors and $\epsilon>0$. 
The proximity operator of $g:= \{\lambda_a\} \log(|\cdot| + \epsilon)$ is known explicitly \cite{prater2022proximity} and given by the following:
\begin{equation*}
    (\forall a \in \RR), ~ \prox_{ \tau g}(a) \left\{\begin{array}{ll} 
        0, & \text{ if } |a|<2\sqrt{\tau \lambda_a}-\epsilon \\
        \max\left(0, \mathrm{sign}(a) \frac{|a|-\epsilon + \sqrt{(|a|+\epsilon)^2-4\tau \lambda_a}}{2}\right) & \text{ if } |a|=2\sqrt{\tau \lambda_a}-\epsilon  \\
        \mathrm{sign}(a)\frac{|a|-\epsilon + \sqrt{(|a|+\epsilon)^2-4\tau \lambda_a}}{2} & \text{ otherwise }
    \end{array}\right.,
\end{equation*}
The same formula holds for the detail coefficients with $\lambda_a$ replaced by $\lambda_d$. 
This problem is not convex, therefore stochastic BC-PG algorithms are not guaranteed to converge to a solution, but for completeness of the presentation, we include them in our experiments. %

\paragraph{Dataset.} In this section we consider the image of the Cameraman, of size $1024 \times 1024$. We will apply a Gaussian blur and a Gaussian noise to obtain the degraded image. The regularization will be done with a $2$-Level log sum-Haar wavelet. 

\paragraph{Experimental setup.} We compare our algorithm to several versions of BC-PG and to the standard forward-backward algorithm. The block methods all consider four blocks from the wavelet decomposition (see Figure \ref{fig5:two_level_wavelet}), with the first block corresponding to the approximation block and the remaining ones to the detail blocks\footnote{Differently from the previous section, the details are not grouped in a single block $d$, they rather form three separate blocks. }. 
We will consider the three following algorithms as baselines:
\begin{itemize}
    \item {FB}: the forward-backward algorithm. Update rule: 
    \begin{equation*}
        \forall n\in\mathbb{N}, \quad  (\varepsilon^n_0,\varepsilon^n_1, \varepsilon_2^n, \varepsilon_3^n) = (1,1,1,1).
    \end{equation*}
    \item {Cyclic BC-PG}: a cyclic BC-PG algorithm that updates only one block at a time in a cyclic manner. The order of the updates is chosen randomly for one cycle at initialization, and then kept identical for all cycles: 
    \begin{equation*}
        \forall n\in\mathbb{N}, \forall \ell \in\{0,1,2,3\}, \quad \varepsilon^n_\ell = \left\{\begin{array}{cc}
            1 & \text{if } \ell=\sigma(n ~\text{mod }4) \\%  \\
            0 & \text{otherwise}
        \end{array}\right.,
    \end{equation*}
    where $\sigma:\{0,1,2,3\}\to\{0,1,2,3\}$ is a permutation, and $n$ mod $4$ denotes the reminder of the division of $n$ by $4$.
    \item {Random BC-PG}: the BC-PG algorithm with one randomly chosen block updated at each iteration. Let \((\{0,1\}, \mathcal{F}, \mathsf{P})\) be a probability space where $\mathcal{F} = 2^{\{0,1\}}$, and the probability measure \(\mathsf{P}\) is defined by $\mathsf{P}(\{1\}) = p,$ $\mathsf{P}(\{0\}) = 1-p,$
where \(p \in [0,1]\). Then,
    \begin{equation*}
       \forall n\in\mathbb{N}, \exists! \ell \in\{0,1,2,3\},  \quad \varepsilon^n_\ell = 1.
    \end{equation*}
    The probability of activation is uniform across the blocks.
\end{itemize}

Note that FB and Cyclic BC-PG are both included in our framework, while Random BC-PG is not. 

We will also consider some new schemes, which where not theoretically covered until now, and that allow to exploit the structure of the problem:
\begin{itemize}
    \item $\acronym$: the proposed Flexible Block-Coordinate Forward-Backward algorithm, which alternates $m\leq 10$ \textit{coarse} updates on the approximations and $10-m$ full updates. For all $n \in \mathbb{N},$
    \begin{align*}
        & (\forall i=1,\ldots,m),~ (\varepsilon^{10n+i}_1,\varepsilon^{10n+i}_2,\varepsilon^{10n+i}_3,\varepsilon^{10n+i}_4) = (1,0,0,0), \\
        &  (\forall i=m+1,\ldots,10-m),~ (\varepsilon^{10n+i}_1,\varepsilon^{10n+i}_2,\varepsilon^{10n+i}_3,\varepsilon^{10n+i}_4) = (1,1,1,1).
    \end{align*}
    \item {Alternating $\acronym$}: the proposed $\acronym$ algorithm, alternating between $m$ updates on the approximation coefficients and $m-10$ updates on the detail coefficients. For all $n \in \mathbb{N},$
    \begin{align*}
        & \forall i=1,\ldots,m,~ (\varepsilon^{10n+i}_1,\varepsilon^{10n+i}_2,\varepsilon^{10n+i}_3,\varepsilon^{10n+i}_4) = (1,0,0,0) \\
        &  \forall i=m+1,\ldots,10-m,~ (\varepsilon^{10n+i}_1,\varepsilon^{10n+i}_2,\varepsilon^{10n+i}_3,\varepsilon^{10n+i}_4) = (0,1,1,1).
    \end{align*}
    \item Stochastic $\acronym$: a stochastic version of the proposed $\acronym$, with $m\leq 10$ coarse iterations followed by $10-m$ full updates in expectation. The proof of convergence of this algorithm is in Appendix \ref{secA:stochastic}. It is based on a stochastic and parallel BC-PG algorithm, with the selection rule in Lemma \ref{lemma5:V_scheme}.
    For all $n \in \mathbb{N},$ we impose that $\varepsilon^n_2 = \varepsilon^n_3 = \varepsilon^n_4:=\epsilon^n$ and
    \begin{align*}
        & \mathsf{P}(\varepsilon^{n}_1=1) = 1, \\
        &  \mathsf{P}\left(\epsilon^n= 1\right) >0.
    \end{align*}
    
\end{itemize}
In the proposed deterministic hierarchical algorithms, we update all the details coefficients simultaneously. With this choice we intend to show that our block update rule, forcing the update of the approximation coefficients at each iteration, is more efficient than the random one. Note that we also tested a Random BC-PG algorithm that splits the image in four equally sized patches, but the results were worse than all the other algorithms presented here.

\begin{figure}
    \centering
    \includegraphics[trim={0em 0em 0em 1em},clip,width=0.8\textwidth]{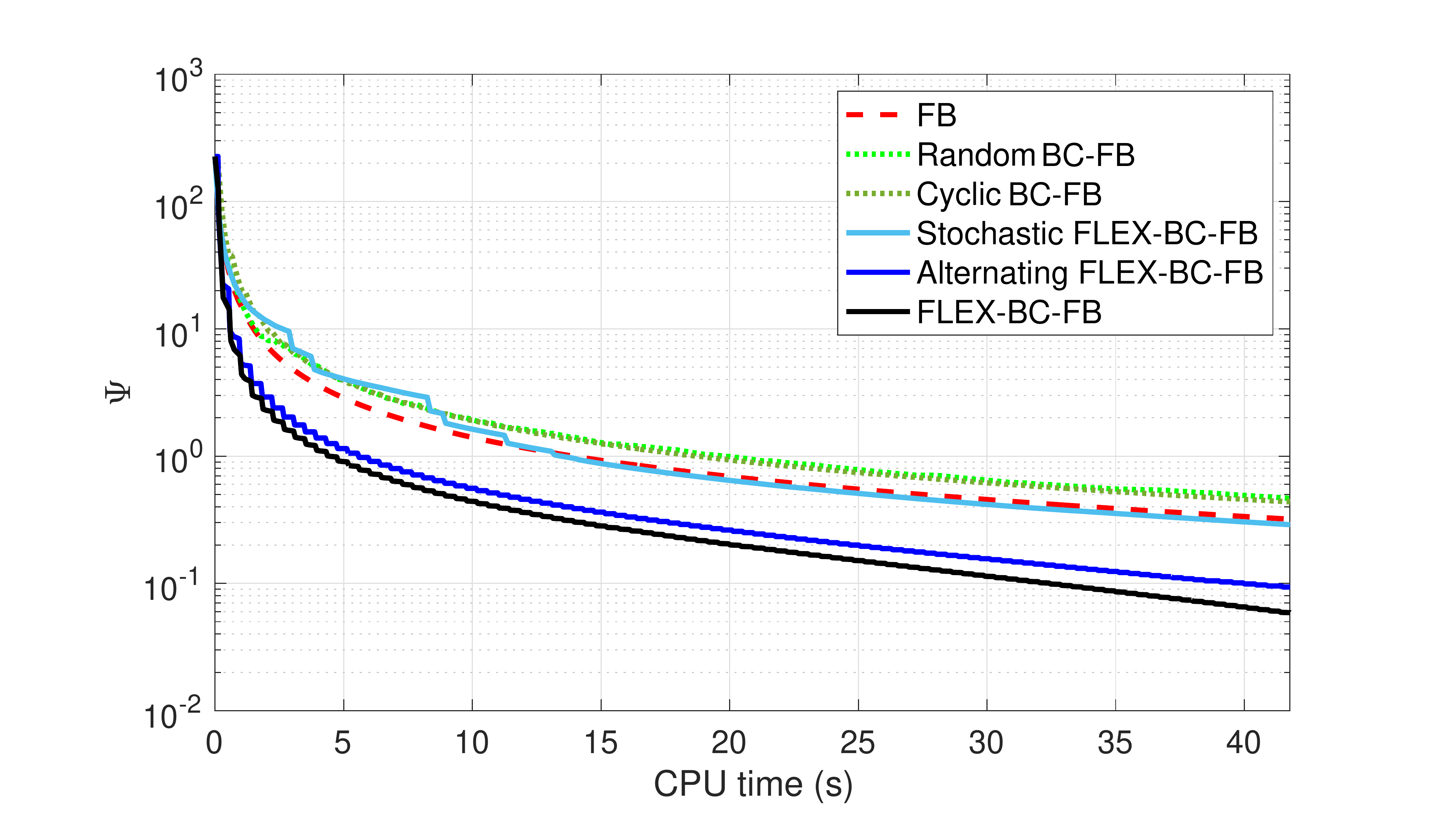} 
    \caption{\label{fig5:BCvsMLvsm}Comparison of the convergence of  $\mathtt{FB}$ (\textcolor{red}{red}), $\mathtt{cyclic ~BC-PG}$ (\textcolor{green}{green}), $\mathtt{random ~ BC-PG}$ (\textcolor{OliveGreen}{dark green}),  $\acronym$ (\textcolor{black}{black}), $\mathtt{Stochastic}$ $\acronym$ (\textcolor{cyan}{cyan}), and $\mathtt{Alternating ~(PI)}$ $\acronym$ (\textcolor{blue}{blue}) for the deconvolution problem regularized with $2$-Level log sum-Haar wavelet on a $1024 \times 1024$ image of the Cameraman. 
    Degradation: 
    Gaussian noise with $\sigma_{\mathrm{noise}} = 0.01$ and a Gaussian blur of size $40\times 40$ and $7$ standard deviation. Parameters choice: $\lambda_a = 1 \times 10^{-10}$, $\lambda_d = 1 \times 10^{-4}$, 
    $m=8$.
    }
\end{figure}

The results of this experiment are shown in Figure \ref{fig5:BCvsMLvsm}. We can clearly see that $\acronym$ and Alternating $\acronym$ vastly outperform the baselines, while the randomized version of $\acronym$ has similar performance to the others BCD algorithms. This difference in performance is due to the fact that our deterministic algorithm can exploit the knowledge of the order of the updates to avoid inefficient back and forth between the blocks, which requires computing an inverse wavelet transform and a wavelet transform to take into account the contribution of the other blocks.

Furthermore, we also compare our update rule to that of the "greedy" literature. By computing the partial gradient with respect to each block at each iteration of our $\acronym$ algorithm, we note that the norm of the partial gradient associated with the approximation coefficients is higher than the norm of the other partial gradients. This gap decreases along the iterations. This indicates that a Gauss-Southwell rule or a Gauss-Southwell-Lipschitz\footnote{The Gauss-Southwell-Lipschitz rule divides the norm of each partial gradient by its associated Lipschitz constant. Here the Lipschitz constant is equal to $1$ for all possible configuration of blocks.} rule \cite{nutini2022let} would behave as our proposed algorithm. This further validates the interest of our framework, since Gauss-Southwell BCD only guarantee the decrease of objective function value, and in the convex case. %

\section{Conclusion}
In this paper we introduce a general block-coordinate forward-backward algorithm, whose convergence is guaranteed in a non-convex setting for a wide range of update rules, encompassing known ones, e.g., cyclic, essentially cyclic, but also new ones, e.g., those inspired by multilevel algorithms. We show on a high dimensional problem that several instances of our $\acronym$ algorithm are competitive with respect to BCD algorithms from the literature and the standard FB algorithm. This general algorithm widens the applicability of BCD approaches, whose update rules  can now fully  exploit  the structure of the problem at hand.

\backmatter

\bmhead{Acknowledgements}
  This work was partially supported by the Fondation Simone et Cino Del Duca - Institut de France and the MEPHISTO (ANR-24-CE23-7039-01) project of the French National Agency for Research (ANR). The work of Luis Briceño-Arias was supported by the National Agence of Research and Development (ANID) from Chile, under the grants FONDECYT 1230257, MATH-AmSud 23-MATH-17, and Centro de Modelamiento Matem\'atico (CMM) BASAL fund FB210005 for centers of excellence. %

\bibliography{references}
\begin{appendices}
\section{Supplementary material for convergence proofs.} \label{secA:proofs_BCD}
In this section, we report the descent lemmas we used to assert the convergence of our $\acronym$ algorithm.
\begin{lemma}{\textbf{Descent lemma \cite{bertsekas1999nonlinear,ortega2000iterative}.}} \label{lm:descent}
Let $f:\RR^N \to \RR$ be a continuously differentiable function with Lipschitz continuous gradient and Lipschitz constant $\beta_f$. Then for any $\beta\geq \beta_f$,
\begin{equation}
    f(x) \leq f(y) + \langle x-y, \nabla f(y) \rangle + \frac{\beta_f}{2}\Vert x-y \Vert^2 \text{ for every } x,y \in \RR^N.
\end{equation}
\end{lemma}
\begin{lemma}{\textbf{(Non-convex) proximal-gradient descent lemma \cite{bolte2014proximal}.}} \label{lm:descent_prox}
    Let $f:\RR^N\to \RR$ be a continuously differentiable function with Lipschitz continuous gradient and Lipschitz constant $\beta_f$. Let $g:\RR^N\to \RR$ be a proper, lower semicontinuous function with $\inf_{\RR^N} g >- \infty$. If
    \begin{equation}
        \label{eq5:fb}
        y \in \prox_{\tau g}(x- \tau \nabla f(x)),
    \end{equation}
    then for any $0<\tau<\frac{1}{\beta_f}$ %
    \begin{equation}
        f(y) + g(y) + \frac{1}{2}\left(\frac{1}{\tau} - \beta_f\right) \Vert x-y \Vert^2 \leq f(x) + g(x).
    \end{equation}
\end{lemma}
    \begin{proof}
    First $\prox_{\tau g}(\cdot)$ is well-defined by \cite[Proposition 2]{bolte2014proximal}.  Thus, for all $\mathbf{x}\in \RR^N$, there exists $y\in  \prox_{\tau g}(x- \tau \nabla f(x))$. This inequality comes directly from \cite[Lemma 2]{bolte2014proximal}, but for completeness of the argument we reproduce it here. 
    By definition of the proximity operator:
    \begin{equation*}
        y \in \argmin_{z\in\RR^N} \langle z-x,\nabla f (x) \rangle + g(z) + \frac{1}{2\tau} \Vert z-x\Vert^2
    \end{equation*}
    Thus taking $z=x$ we obtain
    \begin{align*}
        \langle y-x,\nabla f (x) \rangle + g(y) + \frac{1}{2\tau} \Vert y-x\Vert^2 & \leq \langle x-x,\nabla f (x) \rangle + g(x) + \frac{1}{2\tau} \Vert x-x\Vert^2 \\
        & \leq g(x)
    \end{align*}
    Now invoking Lemma \ref{lm:descent}, we have:
    \begin{equation*}
            \langle \nabla f(x), x-y \rangle \leq f(x)-f(y) + \frac{\beta_f}{2}\Vert x-y \Vert^2
    \end{equation*}
    which yields for any $0<\tau<\frac{1}{\beta_f}$
    \begin{equation*}
        f(y) + g(y) + \frac{1}{2}\left(\frac{1}{\tau} - \beta_f\right) \Vert x-y \Vert^2 \leq f(x) + g(x).
    \end{equation*}
    \end{proof} 
\section{Convergence of a stochastic $\acronym$ for convex optimization} 
\label{secA:stochastic}
In this section, we briefly present a convergence result for a randomized version of our Flexible Block-Coordinate Forward-Backward algorithm. The convergence result in itself is a direct application of \cite[Theorem 4.9]{salzo_parallel_2022}. We aim here to construct a stochastic BC FB that, in expectation, mirrors the behavior of our multilevel algorithm and is convergent. Such algorithm follows classic rules of stochastic BCD algorithms that can update  the blocks in parallel. 

With such algorithm we will be able to have a complete comparison of the update rules available for BC descent algorithms. Recall that the algorithm is of the following form:
Let $(\boldsymbol{\varepsilon}^n)_{n\in \mathbb{N}} = (\varepsilon_1^{n},\ldots,\varepsilon_L^{n})_{n\in \mathbb{N}}$ be a sequence of variables with value in $\{0,1\}^L$. Let $(\tau_\ell)_{1\leq \ell \leq L} \in \RR^{L}_{++}$ and $\mathbf{x}^0 = (x_1^{0},\ldots,x_L^{0}) \in $ dom $g$. Iterate %
\begin{equation}\label{eq5:SBC_FB}
\begin{array}{l}
\text{for } n=0,1,\dots \\ 
\left\lfloor
\begin{array}{l}
\text{for } \ell=1,\ldots,L \\
\left\lfloor 
\begin{array}{l}
    x_\ell^{n+1} = x_\ell^n + \varepsilon_\ell^n \left( \prox_{\tau_\ell g_\ell} \left(x_\ell^n - \tau_\ell \nabla_\ell f(\mathbf{u}^n) \right) - x_\ell^n\right).
\end{array}
\right.
\end{array}
\right.
\end{array}
\end{equation}
The main difference w.r.t. the paradigm proposed in the previous sections is that  $\varepsilon$ and $\mathbf{x}$ are now random variables. Consider the following assumptions: \medskip
\begin{assumption}
    \textcolor{white}{y}
    \begin{itemize}
        \item[\normalfont A$6$] $f:\Hi \rightarrow \RR$ is convex and continuously differentiable,
        \item[\normalfont A$7$] for every $\ell = 1,\ldots, L$, $g_\ell:\Hi_\ell \rightarrow (-\infty,+\infty]$ is proper, convex, and lower semicontinuous.
        \item[\normalfont A$8$] $\boldsymbol{\varepsilon} =(\varepsilon_1,\ldots,\varepsilon_L)$ is a random variable with values in $\{0,1\}^L$, such that for every $\ell \in \{1,\ldots,L\},$ $\mathsf{P}(\varepsilon_\ell = 1)>0$ and $\mathsf{P}(\boldsymbol{\varepsilon}=(0,\ldots,0))=0$.
    \end{itemize}
\end{assumption}
We can now present a way to construct update rules to mimic our multilevel algorithm that verify Assumption A$8$. As multilevel algorithms mostly employ $V$-scheme in practice \cite{fung2020,javaherian2017,parpas2017,lauga_iml_2023,lauga2022fista,lauga2024radio}, we present an update rule for this scheme. \medskip
\begin{lemma}{\textbf{V-scheme probabilities for stochastic $\acronym$.}} \label{lemma5:V_scheme}
    Suppose that $\boldsymbol{\varepsilon} =(\varepsilon_1,\ldots,\varepsilon_{L})$ is a random variable with values in $\{0,1\}^L$, such that $\mathsf{P}(\varepsilon_1 = 1) = 1$ and for every $\ell \in \{1,\ldots,L-1\},$
    \begin{itemize}
        \item $\mathsf{P}(\varepsilon_{\ell+1}=1|\varepsilon_{\ell}=1) >0 $,
        \item $\mathsf{P}(\varepsilon_{\ell+1}=1|\varepsilon_{\ell}=0) = 0$.
    \end{itemize}

    Then, for every $\ell \in \{1,\ldots,L\},$ $\mathsf{P}(\varepsilon_\ell = 1)>0$ and $\mathsf{P}(\boldsymbol{\varepsilon}=(0,\ldots,0))=0$.
\end{lemma}
\begin{proof}
    The second point is straightforward. For the first point, simply remark that for every $\ell \in \{2,\ldots,L\}$:
    \begin{equation*}
        \mathsf{P}(\varepsilon_\ell = 1) = \mathsf{P}(\varepsilon_\ell = 1 | \varepsilon_{\ell-1} = 1) \mathsf{P}(\varepsilon_{\ell-1} = 1),
    \end{equation*}
    then one directly has:
    \begin{equation*}
        \mathsf{P}(\varepsilon_\ell = 1) = \left( \prod_{j=2}^{j=\ell} \mathsf{P}(\varepsilon_j = 1 | \varepsilon_{j-1} = 1) \right)\mathsf{P}(\varepsilon_{1} = 1).
    \end{equation*}
    which is strictly greater than $0$.
\end{proof}

One can see that with this construction we will update the coarsest level at each iteration, and that updating "fine" levels will also force us to update coarser levels, which is typical of multilevel methods.

The sampling of $\boldsymbol{\varepsilon}$ is done sequentially by increasing $\ell$ until we reach the first zero occurrence. In order to update all levels as often as possible, the value of $\mathsf{P}(\varepsilon_{\ell+1}=1|\varepsilon_{\ell}=1)$ should be close to $1$ for large $\ell$. 

\paragraph{Choosing the right value for the conditional probabilities.} In a typical V-scheme, a multilevel algorithm would compute $m$ iterations at each coarse level, going upwards in the resolution. After that it would compute one iteration at fine level. Thus, we should adjust the conditional probabilities of activating each block so that with high probability we update $m\geq 0$ times the coarsest level alone, then $m$ times the coarsest level and the second to last coarsest level, and so on.  We thus impose for all $\ell$:
\begin{equation*}
    \mathsf{P}(\varepsilon_\ell=1) = \left(\frac{1}{m}\right)^\ell,
\end{equation*}
which yields:
\begin{equation*}
    \mathsf{P}(\varepsilon_{\ell+1}=1|\varepsilon_{\ell}=1) = \frac{1}{m}.
\end{equation*}

\paragraph{Convergence of the stochastic algorithm.} We can now state the convergence result for the stochastic version of our algorithm. The proof is a direct application of \cite[Theorem 4.9]{salzo_parallel_2022} and is therefore omitted. We denote by $\mathsf{E}$ the expected value.\medskip
\begin{theorem}{\textbf{Convergence of stochastic $\acronym$ \cite[Theorem 4.9]{salzo_parallel_2022}.}} \label{th5:stochastic}
     Let $(\boldsymbol{\varepsilon}_n)_{n\in \mathbb{N}} = (\varepsilon_1^{n},\ldots,\varepsilon_L^{n})_{n\in \mathbb{N}}$ be a sequence of independent copies of $\boldsymbol{\varepsilon}$. Let $(\tau_\ell)_{1\leq \ell \leq L} \in \RR^{L}_{++}$ and $x_0 = (x_{1,0},\ldots,x_{L,0}) \equiv \mathbf{x}^0 \in $ dom $g$ be a constant random variable.  
Set $\delta = \max_{1\leq \ell \leq L} \tau_\ell \beta_\ell$ (the block Lipschitz constants, see Assumption \ref{ass:deterministic2}) and $\mathsf{p}_{\text{min}} = \min_{1\leq \ell \leq L} \mathsf{P}(\varepsilon_\ell = 1)$. 

Set $\mathbf{Id} = \bigoplus_{\ell=1}^{L} \frac{1}{\tau_\ell \mathsf{P}(\varepsilon_\ell = 1)} \Id_\ell$ (the identity operators on $\Hi_\ell$), $\Psi_* = \inf \Psi,$ and $S_* = \argmin \Psi \subset \Hi$. Then the following hold.
\begin{enumerate}[label=(\roman*)]
    \item $\mathsf{E}[\Psi(\mathbf{x}^n)] \rightarrow \Psi_*$.
    \item Suppose that $S_* \neq \emptyset$. Then $\mathsf{E}[\Psi(\mathbf{x}^n)]- \Psi_* = o(1/n)$ and for every integer $n\geq 1$,
    \begin{equation*}
        \mathsf{E}[\Psi(\mathbf{x}^n)] - \Psi_* \leq \left[\frac{\text{dist}_\mathbf{Id}^2(x_0,S_*)}{2} + \left(\frac{\max\{1,(2-\delta)^{-1}\}}{\mathsf{p}_{\text{min}}}-1\right)(\Psi(\mathbf{x}^0)-\Psi_*)\right]\frac{1}{n}
    \end{equation*}
    Moreover there exists a random variable $x_*$ taking values in $S_*$ such that $\mathbf{x}^n \rightharpoonup x_*$.
\end{enumerate}
\end{theorem}
\section{Implementation details for the numerical experiments}
\subsection{Efficient computation of the gradient of the approximation}
In our numerical experiments, the degradation is a Gaussian blur. As this blur is symmetric, the blurring matrix $\A \in \RR^{N \times N}$ can be expressed as a Kronecker product \cite{hansen2006} 
\begin{equation*}
    \A = \A_r \otimes \A_c
\end{equation*}
where $\A_r$ and $\A_c$ are $\sqrt{N} \times \sqrt{N}$ real matrices that decompose the action of the blur into its vertical ($c$ for columns) and horizontal ($r$ for rows) components. Using the following relationship
\begin{equation*}
    \A x = \textrm{vec}\left( \A_c X \A_r^\top \right)
\end{equation*}
where vec denotes the vectorization, and $X\in \mathbb{R}^{\sqrt{N} \times \sqrt{N}}$ is our image in its matrix form, we can avoid storing $\A$ and exploit a similar relationship to compute the gradient of the "coarse" function w.r.t. the approximation coefficients without having to compute the global gradient.

Indeed, under similar conditions as for the Gaussian blur \cite{lauga_iml_2023}, the projection operation $\Pi_V$ can be written using a Kronecker product. We have \cite[Section 3.2]{lauga_iml_2023} that:
\begin{equation*}
    \Pi_V = \R_{\mathbf{q},r} \otimes \R_{\mathbf{q},c}
\end{equation*}
where $\R_{\mathbf{q},r},\R_{\mathbf{q},c} \in \RR^{\sqrt{N} \times \sqrt{N}/2}$ are Toeplitz matrices generated from the quadrature mirror filter $\mathbf{q}$ \cite[Section 3.2]{lauga_iml_2023} of the wavelet transform that defined $\Pi_V$. For square images, $\R_{\mathbf{q},r} = \R_{\mathbf{q},c} := \R$.

Therefore the gradient of the data fidelity term w.r.t. the approximation coefficients can be expressed as:
\begin{align}
    & \Pi_V \A^* \left( \A\Pi_V^* a^n - \Pi_V^*\Pi_V \mathbf{z} \right)  = \R \A_c^\top \A_c \R^\top a_o \R \A_r^\top \A_r \R^\top - \R \A_c^\top \R^\top \R  z  \R^\top \R \A_r \R^\top  
\end{align}
As $\R \A_c^\top \A_c \R^\top$, 
    $\R \A_r^\top \A_r \R^\top$, 
    $\R \A_c^\top \R^\top$, 
   $\R z \R^\top$, 
    $\R \A_r \R^\top$, 
all belong to $\RR^{N_H,N_H}$ and can be pre-computed, the gradient of $1/2\Vert \A\cdot-\mathbf{z}\Vert_2^2$ w.r.t. the approximation coefficients can be evaluated efficiently.
\end{appendices}

\end{document}